\documentclass[12pt,a4paper,leqno]{amsart}
\usepackage[utf8]{inputenc}
\pdfoutput=1

\usepackage{amssymb,amsfonts,mathtools,mathrsfs}   
\usepackage{graphicx}
\usepackage{placeins}
\usepackage{enumerate}
\usepackage{tikz}
\usepackage{setspace}
\usepackage{xcolor}
\usepackage{hyperref}
\usepackage{theoremref}

\usepackage{dsfont}

\usepackage[
  margin=1.5cm,
  includefoot,
  footskip=30pt,
]{geometry}

\usepackage[
backend=biber,
style=alphabetic,
maxbibnames=99, 
]{biblatex}
\renewcommand{\epsilon}{\varepsilon}
\renewcommand{\leq}{\leqslant}
\renewcommand{\le}{\leqslant}
\renewcommand{\geq}{\geqslant}
\renewcommand{\ge}{\geqslant}

\newtheorem{thm}{Theorem}[section]
\newtheorem{cor}[thm]{Corollary} 
\newtheorem{lem}[thm]{Lemma} 
\newtheorem{prop}[thm]{Proposition}

\theoremstyle{definition}

\newtheorem{rem}[thm]{Remark}

\numberwithin{equation}{section}

\newcommand{\C}{\mathbb{C}}
\newcommand{\R}{\mathbb{R}}

\newcommand{\N}{\mathbb{N}}

\newcommand{\Sph}{\mathbb{S}}

\newcommand{\diam}{\operatorname{diam}}
\newcommand{\dist}{\operatorname{dist}}
\newcommand{\supp}{\operatorname{supp}}

\newcommand{\PV}{\mathrm{P.V.}}

\newcommand{\RE}{\operatorname{Re}}  
\newcommand{\dd}{\mathop{}\!d} 

\def\XXint#1#2#3{{\setbox0=\hbox{$#1{#2#3}{\int}$}
\vcenter{\hbox{$#2#3$}}\kern-.5\wd0}}

\newcommand{\pa}{\partial}
\newcommand{\ol}{\overline}
\newcommand{\Om}{\Omega}
\newcommand{\Ga}{\Gamma}
\newcommand{\na}{\nabla}
\newcommand{\al}{\alpha}

\newcommand{\ga}{\gamma} 
\newcommand{\chf}{\mathds{1}}
\newcommand{\de}{\delta}
\newcommand{\ka}{\kappa}
\newcommand{\la}{\lambda}
\newcommand{\La}{\Lambda}
\newcommand{\ul}{\underline}
\newcommand{\sS}{\mathcal{S}}

\begin{document}

\title[Quantitative stability for the nonlocal overdetermined Serrin problem]{Quantitative stability \\ for the nonlocal overdetermined Serrin problem}

\author[S. Dipierro]{Serena Dipierro}
\address{Serena Dipierro: Department of Mathematics and Statistics, The University of Western Australia, 35 Stirling Highway, Crawley, Perth, WA 6009, Australia}
\email{serena.dipierro@uwa.edu.au}

\author[G. Poggesi]{Giorgio Poggesi}
\address{Giorgio Poggesi: School of Mathematical Sciences, Adelaide University, Adelaide, SA 5005, Australia}
\email{giorgio.poggesi@adelaide.edu.au}

\author[J. Thompson]{Jack Thompson}
\address{Jack Thompson: Department of Mathematics and Statistics, The University of Western Australia, 35 Stirling Highway, Crawley, Perth, WA 6009, Australia}
\email{jack.thompson@research.uwa.edu.au}

\author[E. Valdinoci]{Enrico Valdinoci}
\address{Enrico Valdinoci: Department of Mathematics and Statistics, The University of Western Australia, 35 Stirling Highway, Crawley, Perth, WA 6009, Australia}
\email{enrico.valdinoci@uwa.edu.au}

\subjclass[2020]{35R11, 35N25, 35B35, 35B50, 47G20} 

\date{\today}

\dedicatory{}

\begin{abstract}
We establish quantitative stability for the nonlocal Serrin overdetermined problem, via the method of the moving planes.
Interestingly, our stability estimate is even better than those obtained so far in the classical setting (i.e., for the classical Laplacian) via the method of the moving planes.

A crucial ingredient is the construction of a new antisymmetric barrier, which allows a unified treatment of the moving planes method. This strategy allows us to establish a new general quantitative nonlocal maximum principle for antisymmetric functions, leading to new quantitative nonlocal versions of both the Hopf lemma and the Serrin corner point lemma.

All these tools -- i.e., the new antisymmetric barrier, the general quantitative nonlocal maximum principle, and the quantitative nonlocal versions of both the Hopf lemma and the Serrin corner point lemma -- are of independent interest. 
\end{abstract}

\maketitle

\section{Introduction}

The classical Serrin problem~\cite{MR333220} is the paradigmatic example of an ``overdetermined'' problem in which the existence of a solution satisfying both a Dirichlet and a Neumann boundary condition necessarily forces the domain to be a ball. 

To efficiently utilize classification results of this type in concrete applications it is often desirable to have a ``quantitative'' version, for instance stating that if the Neumann condition is ``almost'' satisfied, then the domain is necessarily ``close'' to a ball.

The recent literature has also intensively investigated the Serrin problem in the fractional setting. This line of research has profitably developed many aspects of the theory of nonlocal equations, including boundary regularity, Pohozaev identities, maximum and comparison principles, various versions of boundary and external conditions, a specific analysis of antisymmetric functions, etc.

However, despite the broad literature on the fractional Serrin problem, a quantitative version of this result is still missing, due to the intrinsic difficulties in adapting the classical techniques to the nonlocal case.

This paper aims to fill this gap in the literature, finally obtaining a quantitative stability result for the fractional Serrin problem. The result itself will rely on several auxiliary results of independent interest, such as
a new general quantitative nonlocal maximum principle for antisymmetric functions, which will lead to new quantitative nonlocal versions of both the Hopf lemma and the Serrin corner point lemma.
Also, a very fine analysis for a barrier function will be put forth, relying on suitable asymptotics of special functions and bespoke geometric estimates.

Specifically, the setting that we consider here is as follows.

\medskip

Let~$n$ be a positive integer and~$ \Om $ be an open bounded subset of~$\R^n$ with smooth boundary.
For~$s \in (0,1)$, let~$(-\Delta)^s$ denote the fractional Laplacian defined by 
\begin{equation}\label{eq:def fractional Laplacian}
(-\Delta)^s u(x) = c_{n,s} \, \PV \int_{\R^n} \frac{u(x)-u(y)}{\vert x - y \vert^{n+2s}} \dd y ,
\end{equation}
where~$\PV$ denotes the Cauchy principal value and~$c_{n,s}$ is a positive normalization constant.

Let~\(u\) be a solution (in an appropriate sense) of the equation
\begin{equation}\label{eq:Dirichlet problem}
\begin{cases}
(-\Delta)^s u =f(u) \quad & \text{ in } \Omega ,\\
u=0 \quad & \text{ in } \R^n \setminus \Omega, \\
u \geq 0 \quad & \text{ in } \R^n ,
\end{cases}
\end{equation}
where~$f:\R \to \R$ is a locally Lipschitz function satisfying~$f(0)\geq 0$. For~$x\in\pa\Om$, we define the fractional normal derivative of~$u$ at~$x\in\pa\Om$ as
\begin{equation}\label{eq:def fracnormal derivative}
\pa_\nu^s u (x):= \lim_{t \to 0^+ } \frac{u(x) - u(x - t \nu(x))}{t^s},
\end{equation}
where~$\nu$ denotes the exterior unit normal to~$\pa\Om$.

The main result in~\cite{MR3395749} states that, if
the overdetermined problem given by~\eqref{eq:Dirichlet problem} and the Serrin-type condition
\begin{align}
\pa_\nu^s u = \text{constant} \quad \text{ on } \partial \Om \label{eq:overdetermination}
\end{align}
admits a solution, then~\(\Omega\) is necessarily a ball.

Setting~$\de(x)$ to be a smooth function in~$\Om$ that coincides with~$\mathrm{dist}(x, \R^n \setminus \Om)$ in a neighborhood of~$\pa\Om$, we will assume that the solution~$u$ to~\eqref{eq:Dirichlet problem} is such that
\begin{equation}\label{eq:higher boundary regularity}
\frac{u}{\de^s} \in C^1 (\ol{\Om}).
\end{equation}
Such an assumption is automatically satisfied whenever~$\Om$ and~$f$ are smooth enough: indeed, \cite{grubb2014local, grubb2015fractional} prove that if~$\Om$ is smooth and~$f \in C^{\ga}$ then~$\frac{u}{\de^s} \in C^{s+\ga} (\ol{\Om})$ for all~$\ga>0$ (see also~\cite{ros2016boundary} where the~$C^{s+\ga} (\ol{\Om})$ regularity is achieved for~$C^{2,\ga}$ domains, for~$\ga$ small).

We stress that, for~$s=1$, \eqref{eq:higher boundary regularity} recovers the assumption~$u\in C^2(\ol{\Om})$ required in the classical proof established by Serrin via the method of moving planes: in particular, that assumption was needed to obtain a contradiction to Serrin corner lemma. In the same way, we believe that, in the nonlocal setting~$s\in(0,1)$, the assumption in~\eqref{eq:higher boundary regularity} is needed to obtain a contradiction to the nonlocal version of Serrin corner lemma and hence the rigidity result (see the second case in the proof of the forthcoming \thref{lem:symmetric difference}).
Note, in the original version of~\cite{MR3395749}, assumption~\eqref{eq:higher boundary regularity}
was overlooked.
For a further discussion on the assumption~\eqref{eq:higher boundary regularity}, see \cite[Remark 3.1]{jarohs_qualitative_2025}.

In this paper, we obtain quantitative stability for such a rigidity result.

We measure how close~\(\Omega\) is to being a ball with the quantity
\begin{equation*}
\rho(\Omega) := \inf \big\{  \rho_e - \rho_i \text{ : } B_{\rho_i} (x) \subset \Omega \subset B_{\rho_e} (x) \text{ for some } x \in \Omega \big\},
\end{equation*}
and how close~\( \pa_\nu^s u\) is to being constant on~\(\partial \Omega\) with the Lipschitz semi-norm 
\begin{equation}\label{eq:def Lipschitz seminorm}
[ \pa_\nu^s u]_{\pa \Om} := \sup_{\substack{x,y\in \pa \Om \\ x\neq y}}\frac{\vert \pa_\nu^s u(x) - \pa_\nu^s u(y) \vert}{\vert x - y \vert }.
\end{equation}

Moreover, we consider the following weighted~$L^1$-norm
\begin{equation}\label{semiliweigh}
\| u \|_{L_s(\R^n)}  = \int_{\R^n} \frac{ \vert u(x) \vert }{ 1+ \vert x \vert^{n+2s} } \dd x .\end{equation}

Then, our main result is the following:

\begin{thm} \thlabel{thm:Main Theorem}
Let~\( \Om \) be an open bounded subset of~$ \R^n$ of class~$C^2$ satisfying the uniform interior sphere condition\footnote{As usual, we say that~$\Om$ satisfies the \emph{uniform interior sphere condition with radius~$r_\Om$} if for any point~$x\in\pa\Om$ there exists a ball~$B_{r_\Om}\subset\Om$ of radius~$r_\Om$ such that~$\pa\Om \cap \pa B_{r_\Om} = \left\lbrace x \right\rbrace$. In particular, if \(\Omega\) is a bounded open set with \(C^2\) boundary then it satisfies the uniform interior sphere condition.} with radius~$ r_\Om >0$.
Let~$f \in C^{0,1}_{\mathrm{loc}}(\R)$ be such that~$f(0)\geq 0$.
Let~\(u\) be a weak solution of~\eqref{eq:Dirichlet problem} satisfying~\eqref{eq:higher boundary regularity}. 

Define
\begin{equation}\begin{split}\label{eq:def R}
&R:= \min \left\lbrace r_\Om,  \frac{  \ka_{n,s}^{\frac{1}{2s}} |B_1|^{-\frac{1}{n}} }{[f]_{C^{0,1}([0,\|u\|_{L^\infty(\Omega)}])}^{\frac{1}{2s}} } \right\rbrace \\
{\mbox{with }} \qquad
& \ka_{n,s} := 
\frac{n}{2^{1-2s}} |B_1|^{1+2s/n} (1-s) \pi^{-n/2} \frac{\Ga(\frac{n}{2} +s )}{ \Ga(2+s) }
.\end{split}\end{equation}

Then,
\begin{equation}\label{eq:main stability estimate}
\rho(\Omega) \leq C [ \pa_\nu^s u]_{\partial \Om}^{\frac 1 {s+2}} 
\end{equation}
with 
\begin{equation*}
C := C(n,s) \left[1+
\frac{ 1 }{ r_\Om^{1-s}  \left( f(0)+\|u\|_{L_s(\R^n)} \right) } \right]
\left( \frac{\diam \Omega}{R} \right)^{2n+3+6s} 
R^{n+2-s} .
\end{equation*}
\end{thm}

An immediate consequence of \thref{thm:Main Theorem} is that we recover the main result in~\cite{MR3395749}.

\begin{cor} \thlabel{50v37N6a}
Let~\( \Om \) be an open bounded subset of~$ \R^n$ of class~$C^2$, and~$f \in C^{0,1}_{\mathrm{loc}}(\R)$ be such that~$f(0)\geq 0$. If~\(u\) is a weak solution of~\eqref{eq:Dirichlet problem} satisfying~\eqref{eq:higher boundary regularity} and the overdetermined condition~\eqref{eq:overdetermination} then \(\Omega\) is a ball.
\end{cor}

We remark that the assumption \(f(0)\geq 0\) is not necessary for the main result in~\cite{MR3395749} and, if one follows the proof of \thref{thm:Main Theorem}, we also do not require this assumption to establish the rigidity of the nonlocal Serrin problem. However, we do use this assumption to establish \thref{thm:Main Theorem}, so we opted to leave this assumption in the statement of \thref{50v37N6a} so that it is a clear corollary of our main result. \medskip

It is worth noting that the stability estimate in~\eqref{eq:main stability estimate} is
even better than those obtained so far in the classical setting (i.e., for~$s=1$) via the method of moving planes. Indeed, \eqref{eq:main stability estimate} gives H\"older-type stability, which is better than the logarithmic-type stability achieved in~\cite{MR1729395}. We mention that, for~$s=1$, the estimate in~\cite{MR1729395} was improved to H\"older stability in~\cite{MR3522349} at the cost of restricting the analysis to a particular class of domains (which includes convex domains). We stress that, in contrast with~\cite{MR3522349}, \eqref{eq:main stability estimate} holds without restricting the class of domains considered; moreover, in the formal limit\footnote{A careful analysis of the constants in the proof of \thref{thm:Main Theorem} reveals that \(C(n,s)\to+\infty\) as \(s\to 1^-\), so this discussion regarding the limiting exponent can only be understood formally.} for~$s \to 1$, the stability exponent in~\eqref{eq:main stability estimate} is better than that obtained in~\cite{MR3522349}. 

We mention that in the particular case~$s=1$ and~$f \equiv const.$ other stability results have been established with alternative approaches to the method of moving planes; these are mainly inspired by the approach towards Serrin's symmetry result pioneered by Weinberger in~\cite{MR333221} and can be found, e.g., in~\cite{brandolini2008stability, feldman2018stability, MR4124125, MR4054869, MP2023interpolating}.
\medskip

The quantitative version of the fractional Serrin problem has remained completely untouched, though several related problems have been considered under various perspectives: these include the cases in which the Neumann condition in the fractional Serrin problem is replaced by a parallel surface condition,
see~\cite{MR4577340, RoleAntisym2022, DPTVParallelStability},
and some of the methodologies share some features with quantitative versions of fractional geometric problems, such as~\cite{MR3836150}.

We stress that our result is new even in the case~$f \equiv const.$

\begin{rem}
Leveraging the fine analysis provided in \cite[Sections 6 and 1.3]{DPTVParallelStability} -- in particular, reasoning as in \cite[Theorem 6.9]{DPTVParallelStability} -- the stability exponent $1/(s+2)$ of \thref{thm:Main Theorem} may be improved to $\al / ( 1 +\al(s+1) )$, provided that $\Om$ is of class $C^\al$ for $\al>1$. In this case, the inequality
$$
\rho(\Omega) \leq C [ \pa_\nu^s u]_{\partial \Om}^{\al / ( 1 +\al(s+1) )}
$$
holds true with a constant $C$ depending on the same parameters of the constant appearing in \thref{thm:Main Theorem}
%
%
as well as on the $C^\al$ regularity of $\Om$.
\end{rem}

\smallskip

A crucial ingredient to establish \thref{thm:Main Theorem} is the following result, which provides a new powerful antisymmetric barrier that allows a unified treatment of the method of the moving planes. Such a barrier will allow us to establish a new general nonlocal maximum principle (\thref{Quantitative Maximum Principle}) that leads to quantitative versions (Theorems~\ref{thm:quantitative Hopf} and~\ref{thm:quantitative corner lemma}) of both the nonlocal Hopf-type lemma and the nonlocal Serrin corner point lemma established in Proposition~3.3 and Lemma~4.4
of~\cite{MR3395749}.

\begin{thm}\thlabel{prop:newantisymmetricbarrier}
Let~\(s\in (0,1)\), \(n\) be a positive integer, \( a = (a_1 , \dots , a_n) \in \R^n \) with~$a_1>0$, and~\(\rho>0\).
For~$x_0 \in \R^n$, consider
$$
\psi_{B_\rho (x_0)} := \ga_{n,s} \left( \rho^2 - |x-x_0|^2 \right)_+^s \quad \text{ where } \quad
\ga_{n,s}:= \frac{4^{-s} \Ga(\frac{n}{2})}{\Ga(\frac{n}{2}+s)(\Ga(1+s))} ,
$$
and set
\begin{align}\label{eq:def new barrier}
\varphi(x):= x_1  \big (  \psi_{B_\rho(a)} (x)+ \psi_{B_\rho(a_\ast)} (x)  \big ) ,
\end{align}
where~$a_\ast:= a - a_1 e_1$ is the reflection of~$a$ across~$\left\lbrace x_1 = 0 \right\rbrace$.

Then, we have that 
\begin{align}
(-\Delta)^s \varphi (x) \leqslant \frac{n+2s}n  x_1 \qquad \text{in } B_\rho^+(a) \setminus \partial B_\rho(a_\ast). \label{eTrwiP0v}
\end{align} 
\end{thm}

Here above and in the rest of the paper, we denote by~$B_\rho^+(a):=B_\rho(a)\cap \{x_1>0\}$.

\smallskip 

The paper is organised as follows.
In Section~\ref{sec:tools}, we establish \thref{prop:newantisymmetricbarrier}, the new quantitative nonlocal maximum principle (\thref{Quantitative Maximum Principle}), and the quantitative nonlocal versions of Hopf lemma and Serrin corner point lemma (\thref{thm:quantitative Hopf} and \thref{thm:quantitative corner lemma}). These are the crucial tools to achieve \thref{thm:Main Theorem} and are of independent interest. In particular, Section~\ref{subsec:new barrier} is devoted to establish \thref{prop:newantisymmetricbarrier}, which is then used in 
Section~\ref{subsec:Hopf e Serrin Corner} to achieve \thref{Quantitative Maximum Principle}, \thref{thm:quantitative Hopf}, and \thref{thm:quantitative corner lemma}.
Finally, in Section~\ref{sec:main result}, we complete the proof of \thref{thm:Main Theorem}.

\section{Quantitative nonlocal versions of Hopf lemma and Serrin corner point lemma}\label{sec:tools}
\subsection{Notation and setting}
In this section, we fix the notation and provide some relevant definitions.
For~$s\in (0,1)$ and~$n \geq 1$, 
we denote by
\begin{equation*}
[u ]_{H^s(\R^n)} := c_{n,s} \iint_{\R^{2n}} \frac{\vert u(x) - u(y) \vert^2}{\vert x-y\vert^{n+2s}}  \dd x \dd y 
\end{equation*}
the Gagliardo semi-norm of~$u$ and by
\begin{equation*}
H^s(\R^n)  := \big\{ u \in L^2(\R^n) \text{ such that } [u]_{H^s(\R^n)} < +\infty \big\}
\end{equation*}
the fractional Sobolev space.
The positive constant~$c_{n,s}$ is the same appearing in the definition of the fractional Laplacian in~\eqref{eq:def fractional Laplacian}. 

Functions in~$H^s(\R^n)$ that are equal almost everywhere are identified. 

The bilinear form~$\mathcal E : H^s(\R^n) \times H^s(\R^n) \to \R$ associated with~$[\cdot ]_{H^s(\R^n)}$ is given by
\begin{equation*}
\mathcal E(u,v) = \frac{c_{n,s}}{2} \iint_{\R^{2n}} \frac{( u(x) - u(y) )(v(x)- v(y)) }{\vert x-y\vert^{n+2s}} \dd x \dd y .
\end{equation*}
Moreover, recalling~\eqref{semiliweigh}, we define the space
\begin{align*}
L_s(\R^n)  := \big\{ L^1_{\mathrm{loc}} (\R^n) \text{ such that } \| u\|_{L_s(\R^n)} <+\infty \big\}. 
\end{align*}

Let~$ \Omega$ be an open, bounded subset of~$\R^n$ and define the space
\begin{equation*}
\mathcal H^s_0(\Omega) := \big\{ u \in H^s(\R^n) \text{ such that } u =0 \text{ in } \R^n \setminus \Omega \big\}. 
\end{equation*}
Consider~$c\in L^\infty(\Omega)$ and~$g\in L^2(\Omega)$. We say that a function~$u\in  H^s(\R^n) $ is a \emph{weak supersolution} of ~$(-\Delta)^s u + c  u = g$ in~$\Omega$, or, equivalently, that~$u$ satisfies
\begin{equation}\label{eq:def supersol 1}
(-\Delta)^s u + c  u \geq g\quad {\mbox{ in }}\Omega
\end{equation} 
in the \emph{weak sense}, if
\begin{equation}\label{eq:def supersol 2}
\mathcal E(u,v) + \int_\Omega c uv \dd x \geq \int_\Omega g v \dd x  ,
\quad \text{ for all } v\in \mathcal H^s_0(\Omega)\text{ with } v\geq 0 .
\end{equation}
Similarly, the notion of \emph{weak subsolution} (respectively, \emph{weak solution}) is understood by replacing the~$\geq$ sign in~\eqref{eq:def supersol 1}-\eqref{eq:def supersol 2} with~$\leq$ (respectively, $=$).
Notice that we are assuming a priori that weak solutions are in~$L_s(\R^n)$.  
\smallskip

Setting~$Q_T : \R^n \to \R^n$ to be the function that reflects~$x$ across the plane~$T$, we say that a function~$v: \R^n \to \R$ is \emph{antisymmetric with respect to~$T$} if
\begin{equation*}
v(x)  = - v(Q_T(x)) \quad \text{ for all } x \in \R^n .
\end{equation*}
To simplify the notation, we sometimes write~\( x_* \) to denote~$Q_T(x)$ when it is clear from context what~$T$ is.

Often, we will only need to consider the case~$T= \left\lbrace x_1=0 \right\rbrace $, in which case, for~$x=(x_1, \dots,x_n) \in \R^n$, we have that~$x_* = Q_T(x) = x - 2 x_1 e_1 $. 
For simplicity, we will refer to~$v$ as \emph{antisymmetric} if it is antisymmetric with respect to~$ \left\lbrace x_1=0 \right\rbrace $. 

We set
$$\R^n_+:= \big\{ x=(x_1,\dots,x_n)\in\R^n \, \text{ : } \, x_1>0 \big\} $$ and, for any~$A\subseteq\R^n$, we define~$A^+:= A\cap\R^n_+$.

As customary, we also denote by~$u^\pm$ the positive and negative parts of a function~$u$, that is
$$ u^+(x):=\max\{ u(x),0\} \qquad{\mbox{and}}\qquad
u^-(x):=\max\{- u(x),0\}.$$

\smallskip

Given~$ A \subset \R^n\), we define the characteristic function~$\chf_A : \R^n \to \R $ of~$A$ and the distance function~$\delta_A : \R^n \to [0,+\infty]$ to~\(A\) as
\begin{equation*}
\chf_A(x) :=
\begin{cases}
1, &\text{ if } x \in A , \\
0, &\text{ if } x \not \in A ,
\end{cases}
\quad \qquad \text{and} \quad \qquad
\delta_A(x) := \inf_{y \in A} \vert x-y\vert. 
\end{equation*}

Given an open bounded smooth set~$A\subset\R^n$, $\psi_A$ will denote the (unique) function in~$C^s(\R^n) \cap C^\infty(A)$ satisfying
\begin{equation}\label{eq:torsion in A}
\begin{cases}
(-\Delta)^s \psi_A =1 \quad & \text{in } A ,\\
\psi_A =0  \quad & \text{in } \R^n \setminus A . 
\end{cases}
\end{equation} 
For~$A=B_{\rho}(x_0)$, \cite{getoor1961first, bogdan2010heat} show that the explicit solution of~\eqref{eq:torsion in A} is
\begin{equation*}
\psi_{B_\rho (x_0)}:= \ga_{n,s} \left( \rho^2 - |x-x_0|^2 \right)_+^s \quad \text{ with } \quad
\ga_{n,s}:= \frac{4^{-s} \Ga(\frac{n}{2})}{\Ga(\frac{n}{2}+s)(\Ga(1+s))} .
\end{equation*}

Finally, \(\lambda_1(A)\) will denote the first Dirichlet eigenvalue of the fractional Laplacian, that is,
\begin{equation*}
\lambda_1(A) = \min_{u \in \mathcal H^s_0 (A)} \frac{\mathcal{E}(u,u)}{\displaystyle \int_{A} u^2 dx} ,
\end{equation*}
and we will use that
\begin{equation}\label{eq:Faber-Krahn}
\lambda_1(A) \ge \ka_{n,s} |A|^{- \frac{2s}{n}} \quad \text{ with } \quad
\ka_{n,s} := 
\frac{n}{2^{1-2s}} |B_1|^{1+2s/n} (1-s) \pi^{-n/2} \frac{\Ga(\frac{n}{2} +s )}{ \Ga(2+s) }
\end{equation}
for any open bounded set~$A$, see, e.g., \cite{MR3395749,yildirim2013estimates}
(see also~\cite{brasco2020quantitative}).

\subsection{A new antisymmetric barrier: proof of \texorpdfstring{\thref{prop:newantisymmetricbarrier}}{Proposition 1.3}}\label{subsec:new barrier}

This section is devoted to the proof of \thref{prop:newantisymmetricbarrier}.
To this aim, we first prove the following three lemmata. Recall that, given \(a,b,c,z\in \C\), the hypergeometric function is defined by the \emph{Gauss series}\begin{align*}
    {}_2F_1(a,b;c;z)=\sum_{k=0}^{+\infty} \frac{(a)_k(b)_k}{(c)_k} \frac{z^k}{k!} \qquad \text{for all }\vert z \vert<1,
\end{align*} where \((x)_k\) denote the \emph{Pochhammer symbols} defined by \begin{align*}
    (x)_k = \begin{cases}
        x(x+1)\cdots(x+k-1), &\text{if }k\in \N\setminus\{0\}, \\
        1, &\text{if }k=0.
    \end{cases}
\end{align*} When \(c-a-b>0\), the Gauss series converges absolutely and when \(c-b-a\in (-1,0]\), it converges conditionally. For more details, see \cite[Chapter 15]{abramowitz_handbook_1964}.

\begin{lem} \thlabel{P7i5W6Bg}
Let~\(s\in (0,1)\), \(n\) be a positive integer, and
$$ \psi(x) := \gamma_{n,s} (1-\vert x \vert^2)^s_+ \qquad {\mbox{ with }}\quad \gamma_{n,s}:=\frac{4^{-s} \Gamma(n/2)}{\Gamma( \frac{n+2s} 2  ) \Gamma(1+s)}.$$ 

Then, \begin{align*}
(-\Delta)^s \psi(x) &= \begin{cases}
1, &\text{if } x\in B_1, \\
{\displaystyle -a_{n,s} \vert x \vert^{-n-2s } {}_2F_1 \bigg ( \frac{n+2s} 2  , s+1 ; \frac{n+2s}2+1 ; \vert x \vert^{-2} \bigg )   } , & \text{if } x\in \R^n \setminus \overline {B_1}, \\
-\infty, &\text{if } x\in \partial B_1 ,
\end{cases}
\end{align*} where 
$$ a_{n,s}:=  \frac { s\Gamma(n/2)  } {\Gamma  ( \frac{n+2s}2 +1   )\Gamma  ( 1-s)} .$$
\end{lem}

\begin{proof}
The identity~\((-\Delta)^s\psi =1 \) in~\(B_1\) is well-known in the literature, see~\cite[Section~2.6]{MR3469920}, \cite[Proposition~13.1]{MR3916700}, so we will focus on the case~\(x\in \R^n \setminus B_1\).

Let~\(x\in \R^n \setminus \overline{B_1}\). We have that~\(\psi(x) =\gamma_{n,s} \psi_0(\vert x \vert)\) with~\(\psi_0(\tau) :=  (1- \tau^2)^s_+\), so, by~\cite[Lemma~7.1]{MR3916700}, it follows that \begin{align*}
(-\Delta)^s \psi(x) &= \gamma_{n,s}\vert x \vert^{-\frac n 2 -2s-1} \int_0^\infty t^{2s+1} J_{\frac n 2 -1} (t) I(t) \dd t 
\end{align*} where \begin{align*}
I(t) &:= \int_0^\infty \tau^{\frac n 2 } \psi_0(\tau) J_{\frac n 2 -1 } (t \vert x \vert^{-1} \tau ) \dd \tau 
\end{align*} provided that both these integrals exist and converge.

To calculate~\(I(t)\) we make the change of variables, \(\tau = \sin \theta\) to obtain \begin{align*}
I(t) &= \int_0^1\tau^{\frac n 2 } (1-\tau^2)^s J_{\frac n 2 -1 } (t \vert x \vert^{-1} \tau ) \dd \tau \\
&= \int_0^{\frac \pi 2 }(\sin \theta)^{\frac n 2 } (\cos \theta )^{2s+1} J_{\frac n 2 -1 } (t \vert x \vert^{-1} \sin \theta ) \dd \theta . 
\end{align*} Next, we make use of the identity  \begin{align*}
\int_0^{\frac \pi 2} J_\mu (z\sin \theta) (\sin \theta)^{\mu+1} (\cos \theta)^{2\nu +1} \dd \theta &= 2^\nu \Gamma(\nu+1) z^{-\nu-1} J_{\mu+\nu+1}(z)
\end{align*} which holds provided~\(\RE \mu>-1\) and~\(\RE \nu>-1\), see~\cite[Eq.~10.22.19]{NIST:DLMF}. From this, we obtain \begin{align*}
I(t)&= 2^s \Gamma(s+1) \bigg ( \frac{\vert x \vert } t \bigg )^{s+1} J_{\frac n 2 +s}( \vert x \vert^{-1}t ).
\end{align*} Hence, \begin{align*}
(-\Delta)^s \psi (x) &= 2^s \gamma_{n,s}\Gamma(s+1) \vert x \vert^{-\frac n 2 -s} \int_0^\infty t^s J_{\frac n 2 -1} (t) J_{\frac n 2 +s}( \vert x \vert^{-1}t ) \dd t .
\end{align*} To compute this integral, we recall the identity \begin{align*}
&\int_0^\infty t^{-\lambda} J_\mu(at) J_\nu (bt) \dd t\\
=\;& \frac{a^\mu \Gamma \big( \frac 1 2 \nu + \frac 1 2 \mu - \frac 1 2 \lambda + \frac12\big )}{2^\lambda b^{\mu-\lambda+1}\Gamma \big ( \frac 12 \nu - \frac 1 2 \mu + \frac 12 \lambda + \frac 12 \big ) \Gamma(\mu+1)} {}_2F_1 \bigg ( \frac 12 \big ( \mu + \nu -  \lambda + 1\big )  , \frac12 \big( \mu - \nu -\lambda+1 \big) ; \, \mu+1 ; \, \frac{a^2}{b^2} \bigg )
\end{align*} which holds when~\(0<a<b\) and~\(\RE(\mu+\nu+1)>\RE \lambda >-1\), see~\cite[Eq.~10.22.56]{NIST:DLMF}. 
Using this identity with~$a:=1/|x|$, $b:=$, $\lambda:=-s$, $\mu:=(n+2s)/2$,
and~$\nu:=(n-2)/2$ (and recalling that~\(\vert x \vert >1\)), we obtain that \begin{align*}
(-\Delta)^s \psi (x) &=  \frac {2 \gamma_{n,s} \cdot 4^s \Gamma(s+1)  } {(n+2s)\Gamma  ( -s)}  \vert x \vert^{-n -2s}  {}_2F_1 \bigg ( \frac{n+2s}2  ,s+1; \, \frac{n+2s}2 +1 ; \, \frac1{\vert x \vert^2} \bigg ).
\end{align*} A simple calculation gives that \begin{align*}
 \frac {2 \gamma_{n,s} \cdot 4^s \Gamma(s+1)  } {(n+2s)\Gamma  ( -s)}  &= -\frac { s\Gamma(n/2)  } {\Gamma  ( \frac{n+2s}2 +1   )\Gamma  ( 1-s)} =-a_{n,s}
\end{align*} which completes the case~\(x\in \R^n \setminus \overline{B_1}\). 

Finally, let~\(x\in \partial B_1\). Then, by rotational symmetry, \begin{align*}
(-\Delta)^s \psi(x) &= -\gamma_{n,s} \int_{B_1} \frac{(1-\vert y \vert ^2)^s}{\vert x -y \vert^{n+2s} } \dd y =-\gamma_{n,s} \int_{B_1} \frac{(1-\vert y \vert ^2)^s}{\vert e_1  -y \vert^{n+2s} } \dd y .
\end{align*} Now, let $$
K:= \big\{ y \in \R^n_+ \text{ : } 1/2< y_1 < 1- \vert y' \vert \big\}$$
and notice that~$K\subset B_1$.

In~\(K\), we have that \begin{align*}
\frac{(1-\vert y \vert ^2)^s}{\vert e_1  -y \vert^{n+2s} } &= \frac{(1-y_1^2-\vert y' \vert ^2)^s}{\big ( (1-y_1)^2+\vert y' \vert^2 \big )^{\frac{n+2s}2} } 
\geqslant \frac{(1-y_1^2-(1-y_1)^2)^s}{2^{\frac{n+2s}2}(1-y_1)^{n+2s} } 
\\ &=\frac{\big(2y_1(1-y_1)\big)^s}{2^{\frac{n+2s}2}(1-y_1)^{n+2s} } 
\geqslant  \frac{1}{ 2^{\frac{n+2s}2}(1-y_1)^{n+s}}.
\end{align*} 
Hence,  \begin{align*}
\int_{B_1} \frac{(1-\vert y \vert ^2)^s}{\vert e_1  -y \vert^{n+2s} } \dd y &\geqslant \frac{1}{ 2^{n+2s}} \int_K \frac {\dd y} { (1-y_1)^{n+s}} .
\end{align*} Then the coarea formula gives that \begin{align*}
\int_K \frac {\dd y} { (1-y_1)^{n+s}} &= \int_{1/2}^1 \int_{B^{n-1}_{1-t}} \frac{\dd \mathcal H^{n-1} \dd t }{ (1-t)^{n+s}} \geqslant C \int_{1/2}^1\frac{\dd t}{ (1-t)^{1+s}} = +\infty 
\end{align*} which shows that~\((-\Delta)^s \psi(x)=-\infty\).
\end{proof}

\begin{lem} \thlabel{Fov7scJw}
Let~\(s\in (0,1)\), \(n\) be a positive integer, and 
$$\psi(x) := \gamma_{n,s} (1-\vert x \vert^2)^s_+ \qquad
{\mbox{ with }}\quad \gamma_{n,s}:= 
\frac{4^{-s} \Gamma(n/2)}{\Gamma( \frac{n+2s} 2  ) \Gamma(1+s)}.$$ 
Let~\(p:\R^n \to \R\) be a homogeneous harmonic polynomial of degree~\(\ell\geqslant0\), that is, \(p\) is a polynomial such that \(p(ax)=a^\ell p(x)\) for all~\(a>0\) and~\( x\in \R\) and~\(\Delta p=0\) in~\(\R^n\). 

Then, \begin{align*}
&(-\Delta)^s (p\psi)(x)\\
 =\;& \frac{\gamma_{n,s}}{\gamma_{n+2\ell,s}}  p(x)  \begin{cases}
1, &\text{if } x\in B_1, \\
{\displaystyle -a_{n+2\ell,s} \vert x \vert^{-n-2s-2\ell } {}_2F_1 \bigg ( \frac{n+2s} 2+\ell  , s+1 ; \frac{n+2s}2+1+\ell ; \vert x \vert^{-2} \bigg )   } , & \text{if } x\in \R^n \setminus \overline {B_1}, 
\end{cases}
\end{align*} where $$a_{n,s}:=  \frac { s\Gamma(n/2)  } {\Gamma  ( \frac{n+2s}2 +1   )\Gamma  ( 1-s)} .$$
\end{lem}

\begin{proof}
Let~\(\widetilde \psi :\R^{n+2\ell} \to \R\) be given by \begin{align*}
\widetilde \psi (z) &= \gamma_{n+2\ell,s} (1-\vert z \vert^2 )^s_+ \qquad \text{for all } z\in \R^{n+2\ell}. 
\end{align*} Since~\(\psi\) is a radial function, it follows from~\cite[Proposition~3]{MR3640641} that \begin{align*}
(-\Delta)^s (p\psi)(x) &= \frac{\gamma_{n,s}}{\gamma_{n+2\ell,s}}  p(x) (-\Delta)^s \widetilde \psi (\widetilde x)
\end{align*} where~\(\widetilde x\in \R^{n+2\ell}\) with~\(\vert \widetilde x \vert = \vert x\vert\). Then the result follows from~\thref{P7i5W6Bg} applied in~\(\R^{n+2\ell}\). 
\end{proof}

\begin{lem} \thlabel{TxksrRDI}
Let~\(s\in (0,1)\) and~\(n\) be a positive integer. For all~\(0<\tau <1 \), let \begin{align*}
K(\tau) &:= a_{n,s} \big ( 1- \tau \big )^{-s} \tau^{\frac{n+2s}2 }\\
\text{ and } \quad
F(\tau)&:= {}_2F_1 \bigg (1  , \frac n 2 ; \frac{n+2s}2+1; \tau \bigg )
\end{align*} where $$a_{n,s}:=  \frac { s\Gamma(n/2)  } {\Gamma  ( \frac{n+2s}2 +1   )\Gamma  ( 1-s)} .$$ 
Moreover, let \begin{align*}
f(\tau) &:= 1-K(\tau ) \big ( F(\tau) -1 \big )  \\
\text{ and }\quad
g(\tau )&:= K(\tau) \bigg ( \frac{n+2s}{2s}-F(\tau) \bigg ) -1.
\end{align*} Then~\(f(\tau )\leqslant 1\) and~\(g(\tau) \leqslant 0\), for all~\(0<\tau <1 \). 
\end{lem}

\begin{proof}
The inequality for~\(f\) is trivial since~\(K\geqslant 0\) and, from the definition of the hypergeometric function, we have that~\( F(\tau) \geqslant 1 \). 

To show that~\(g\leqslant 0\), observe first of all that \begin{align*}
 F(1) = \frac{\Gamma \big (\frac{n+2s}2+1 \big ) \Gamma (s )  }{\Gamma \big (\frac{n+2s}2 \big ) \Gamma (s+1)} = \frac{n+2s}{2s}
\end{align*} by~\cite[Eq 15.4.20]{NIST:DLMF}. Hence, by the Fundamental Theorem of Calculus,\begin{align*}
 \frac{n+2s}{2s}-F(\tau) &= \int^1_\tau F'(t) \dd t \\
 &= \frac n {n+2s+2 } \int^1_\tau  {}_2F_1 \bigg (2  , \frac n 2+1 ; \frac{n+2s}2+2; t \bigg ) \dd t\\
 &= \frac n {n+2s+2 } \int^1_\tau  \big(1-t \big )^{s-1} {}_2F_1 \bigg (\frac{n+2s}2  , s+1 ; \frac{n+2s}2+2; t \bigg ) \dd t
\end{align*} using~\cite[Eq 15.5.1]{NIST:DLMF} and~\cite[Eq 15.8.1]{NIST:DLMF}. Now, using~\cite[Eq 15.4.20]{NIST:DLMF} again, we have that \begin{align*}
{}_2F_1 \bigg (\frac{n+2s}2  , s+1 ; \frac{n+2s}2+2; 1 \bigg )  &= \frac{\Gamma \big (\frac{n+2s}2+2 \big ) \Gamma (1-s) }{\Gamma (2) \Gamma \big ( \frac n2 +1 \big )   } = \frac{s(n+2s+2)}{n a_{n,s}} ,
\end{align*}so \begin{align*}
 \frac{n+2s}{2s}-F(\tau)&=\frac n {n+2s+2 } \int^1_\tau  \big(1-t \big )^{s-1}\bigg [  {}_2F_1 \bigg (\frac{n+2s}2  , s+1 ; \frac{n+2s}2+2; t \bigg )-\frac{s(n+2s+2)}{n a_{n,s}}  \bigg ]  \dd t \\
 &\qquad + \frac 1 {a_{n,s}}(1-\tau)^s . 
\end{align*} Hence, \begin{align*}
&g(\tau)\\
 =\;& \big ( 1- \tau \big )^{-s} \tau^{\frac{n+2s}2 } \frac {n\cdot a_{n,s} } {n+2s+2 } \int^1_\tau  \big(1-t \big )^{s-1}\bigg [  {}_2F_1 \bigg (\frac{n+2s}2  , s+1 ; \frac{n+2s}2+2; t \bigg )-\frac{s(n+2s+2)}{n a_{n,s}}  \bigg ]  \dd t \\
 &\qquad +  \tau^{\frac{n+2s}2 } -1.
\end{align*} Finally, using the monotonicity properties of hypergeometric functions, \begin{align*}
{}_2F_1 \bigg (\frac{n+2s}2  , s+1 ; \frac{n+2s}2+2; t \bigg ) \leqslant {}_2F_1 \bigg (\frac{n+2s}2  , s+1 ; \frac{n+2s}2+2; 1 \bigg )=\frac{s(n+2s+2)}{n a_{n,s}}  
\end{align*} which implies that~\(g(\tau ) \leqslant 0\). 
\end{proof}

We now give the proof of \thref{prop:newantisymmetricbarrier}.

\begin{proof}[Proof of \thref{prop:newantisymmetricbarrier}]
Let~\(\psi (x) := \gamma_{n,s}(1-\vert x \vert^2 )^s_+\) and~\(\overline \psi (x) := \gamma_{n,s} x_1 (1-\vert x \vert^2)^s_+\). Since~\(\psi_{B_\rho(a)}(x) = \rho^{2s} \psi( (x-a)/\rho) \), we see that \begin{align*}
x_1 \psi_{B_\rho(a)}(x) &= \rho^{2s+1} \bigg ( \frac{x_1 -a_1}  \rho \bigg ) \psi \bigg ( \frac{x-a} \rho\bigg ) +\rho^{2s} a_1\psi \bigg ( \frac{x-a} \rho\bigg )  \\
&= \rho^{2s+1} \overline \psi \bigg ( \frac{x-a} \rho\bigg ) +\rho^{2s} a_1\psi \bigg ( \frac{x-a} \rho\bigg )  \\ \text{ and }\quad
x_1 \psi_{B_\rho(a_\ast)}(x) &=\rho^{2s+1} \bigg ( \frac{x_1+a_1}\rho \bigg ) \psi \bigg ( \frac{x-a_\ast} \rho\bigg ) -\rho^{2s} a_1\psi \bigg ( \frac{x-a_\ast} \rho\bigg ) \\
&= \rho^{2s+1} \overline  \psi \bigg ( \frac{x-a_\ast} \rho\bigg ) -\rho^{2s} a_1\psi \bigg ( \frac{x-a_\ast} \rho\bigg ).
\end{align*} Hence, for~\(x\in B_\rho(a)\), it follows immediately from Lemmata~\ref{P7i5W6Bg} and~\ref{Fov7scJw} that \begin{equation}\label{VgjyomNF}\begin{split}
(-\Delta)^s (x_1 \psi_{B_\rho(a)}) (x) &=  \rho (-\Delta)^s \overline \psi \bigg ( \frac{x-a} \rho\bigg ) +a_1(-\Delta)^s \psi \bigg ( \frac{x-a} \rho\bigg )  \\
&= \frac{\gamma_{n,s}} {\gamma_{n+2,s}}  ( x_1-a_1) +a_1  \\
&= \frac{n+2s}n x_1 -\frac{2s}n a_1   
\end{split}\end{equation} using also that \begin{align}
\gamma_{n+2,s}=\frac{4^{-s} \Gamma(n/2+1)}{\Gamma( \frac{n+2s} 2+1  ) \Gamma(1+s)} = \frac n {n+2s} \gamma_{n,s} . \label{YVHFpqS4}
\end{align}  Similarly, if~\(x\in B_\rho(a_\ast)\), then  \begin{align*}
(-\Delta)^s ( x_1 \psi_{B_\rho(a_\ast)})(x)  &=  \rho (-\Delta)^s \overline \psi \bigg ( \frac{x-a_\ast} \rho\bigg ) -a_1(-\Delta)^s \psi \bigg ( \frac{x-a_\ast} \rho\bigg ) \\
&= \frac{\gamma_{n,s}} {\gamma_{n+2,s}} (x_1+a_1) -a_1 \\
&= \frac{n+2s} n x_1 +\frac{2s} n a_1 .
\end{align*} This gives that \begin{align*}
(-\Delta)^s \varphi(x) &= \frac{2(n+2s)} n x_1 \qquad \text{in } B_\rho(a) \cap B_\rho(a_\ast)
\end{align*}which, in particular, trivially implies~\eqref{eTrwiP0v}
for points in~$B_\rho(a) \cap B_\rho(a_\ast)$.

Now, towards the proof of~\eqref{eTrwiP0v}, we will give a formula for~\((-\Delta)^s \varphi(x)\) in the case~\(x\in B^+_\rho(a) \setminus \overline{ B_\rho(a_\ast)}\). Equation~\eqref{VgjyomNF} is still valid in this region, so we can focus on computing~\((-\Delta)^s ( x_1 \psi_{B_\rho(a_\ast)})\). Let~\(y=y(x) := \frac{x-a_\ast}{\rho} \) and note that, since~\(x \not \in \overline{ B_\rho(a_\ast)}\), we have that~\(\vert y \vert >1\) which is important for the upcoming formulas to be well-defined. 

By Lemmata~\ref{P7i5W6Bg} and~\ref{Fov7scJw}, we have that \begin{align*}
(-\Delta)^s ( x_1 \psi_{B_\rho(a_\ast)}) &= -\rho \frac{\gamma_{n,s}}{\gamma_{n+2,s}} y_1 a_{n+2,s} \vert y \vert^{-n-2s-2 } {}_2F_1 \bigg ( \frac{n+2s} 2+1  , s+1 ; \frac{n+2s}2+2; \vert y \vert^{-2} \bigg ) \\
&\qquad +a_1a_{n,s} \vert y \vert^{-n-2s } {}_2F_1 \bigg ( \frac{n+2s} 2  , s+1 ; \frac{n+2s}2+1 ; \vert y \vert^{-2} \bigg ) .
\end{align*} Next, observe that \begin{align*}
a_{n+2,s}=  \frac { s\Gamma(n/2+1)  } {\Gamma  ( \frac{n+2s}2 +2   )\Gamma  ( 1-s)} = \frac n {n+2s+2}a_{n,s}, 
\end{align*} so, using also~\eqref{YVHFpqS4}, we have that \begin{align*}
(-\Delta)^s ( x_1 \psi_{B_\rho(a_\ast)}) &= -a_{n,s}\vert y \vert^{-n-2s-2 }  \bigg [ \rho y_1 \bigg ( \frac{n+2s}{n+2s+2} \bigg )   {}_2F_1 \bigg ( \frac{n+2s} 2+1  , s+1 ; \frac{n+2s}2+2; \vert y \vert^{-2} \bigg ) \\
&\qquad -a_1 \vert y \vert^2 {}_2F_1 \bigg ( \frac{n+2s} 2  , s+1 ; \frac{n+2s}2+1 ; \vert y \vert^{-2} \bigg )  \bigg ] .
\end{align*} Then, via the transformation formula, \begin{align*}
{}_2F_1 ( a  ,b ; c; \tau) &= (1-\tau)^{c-a-b} {}_2F_1 ( c-a  ,c-b  ; c; \tau ) 
\end{align*} which, for~\(\tau \in \R\), holds provided that~\(0<\tau <1\), 
see~\cite[Eq~15.8.1]{NIST:DLMF}, we obtain  \begin{align*}
&(-\Delta)^s ( x_1 \psi_{B_\rho(a_\ast)}) \\
=& -a_{n,s} \big ( 1- \vert y \vert^{-2} \big )^{-s} \vert y \vert^{-n-2s-2 }  \bigg [ \rho y_1 \bigg ( \frac{n+2s}{n+2s+2} \bigg )   {}_2F_1 \bigg (1  , \frac n 2 +1 ; \frac{n+2s}2+2; \vert y \vert^{-2} \bigg ) \\
&\qquad -a_1 \vert y \vert^2 {}_2F_1 \bigg ( 1  , \frac n 2  ; \frac{n+2s}2+1 ; \vert y \vert^{-2} \bigg )  \bigg ]   \\
=& -\vert y \vert^{-2 }K \big (  \vert y \vert^{-2 } \big )   \bigg [ \rho y_1 \bigg ( \frac{n+2s}{n+2s+2} \bigg )   {}_2F_1 \bigg (1  , \frac n 2 +1 ; \frac{n+2s}2+2; \vert y \vert^{-2} \bigg )-a_1 \vert y \vert^2 F \big ( \vert y \vert^{-2} \big )   \bigg ] 
\end{align*} using the notation introduced in~\thref{TxksrRDI}. 

Moreover, we apply the following identity between contiguous functions: \begin{align*}
\frac {b\tau} c  {}_2F_1 ( a  ,b+1 ; c+1; \tau) &= {}_2F_1 ( a  ,b ; c; \tau) - {}_2F_1 ( a-1  ,b ; c; \tau),
\end{align*} 
(used here with~$a:=1$, $b:=n/2$ and~$c:=(n+2s+2)/2$)
see~\cite[Eq~15.5.16\textunderscore5]{NIST:DLMF}, and that
\({}_2F_1 ( 0  ,b ; c; \tau) =1 \), to obtain \begin{align*}
 {}_2F_1 \bigg (1  , \frac n 2 +1 ; \frac{n+2s}2+2; \vert y \vert^{-2} \bigg ) &= \frac{n+2s+2}n \vert y \vert^2 \bigg ( {}_2F_1 \bigg (1  , \frac n 2 ; \frac{n+2s}2+1; \vert y \vert^{-2} \bigg ) -1 \bigg ) \\
 &= \frac{n+2s+2}n \vert y \vert^2 \big (  F \big ( \vert y \vert^{-2} \big ) -1 \big ) .
\end{align*} Hence, it follows that \begin{align*}&
(-\Delta)^s ( x_1 \psi_{B_\rho(a_\ast)}) \\=& -K \big (  \vert y \vert^{-2 } \big )  \bigg [ \rho y_1 \bigg ( \frac{n+2s}n \bigg ) \big ( F \big ( \vert y \vert^{-2} \big ) -1  \big )  -a_1 F \big ( \vert y \vert^{-2} \big )  \bigg ] \\
=& -\bigg ( \frac{n+2s}n \bigg )  K \big (  \vert y \vert^{-2 } \big )  \big ( F \big ( \vert y \vert^{-2} \big ) -1  \big ) x_1   + \frac{2s } nK  \big (  \vert y \vert^{-2 } \big )  \bigg (    \frac{n+2s}{2s} -  F \big ( \vert y \vert^{-2} \big )  \bigg )a_1
\end{align*} using that~\(\rho y_1 =x_1+a_1\). Thus, in~\(x\in B^+_\rho(a) \setminus \overline{ B_\rho(a_\ast)}\), \begin{align*}
(-\Delta)^s \varphi(x) &= \bigg ( \frac{n+2s}n \bigg ) \bigg [ 1 - K \big (  \vert y \vert^{-2 } \big )  \big ( F \big ( \vert y \vert^{-2} \big ) -1  \big ) \bigg ]  x_1  \\
&\qquad \qquad  + \frac{2s } n \bigg [ K  \big (  \vert y \vert^{-2 } \big )  \bigg (    \frac{n+2s}{2s} -  F \big ( \vert y \vert^{-2} \big )  \bigg ) -1  \bigg] a_1.
\end{align*} Finally, by Lemma~\ref{TxksrRDI}, we conclude that \begin{align*}
(-\Delta)^s \varphi(x) &\leqslant \frac{n+2s}n  x_1 \qquad \text{in }  B^+_\rho(a) \setminus \overline{ B_\rho(a_\ast)},
\end{align*}
which completes the proof of the desired result in~\eqref{eTrwiP0v}.
\end{proof}

\subsection{A quantitative nonlocal maximum principle and quantitative nonlocal versions of Hopf lemma and Serrin corner point lemma}\label{subsec:Hopf e Serrin Corner}
The following proposition provides a new quantitative nonlocal maximum principle, which enhances~\cite[Proposition~3.6]{DPTVParallelStability} from interior to boundary estimates. Remarkably, this allows a unified treatment for the quantitative analysis of the method of the moving planes, leading to quantitative versions of both the nonlocal Hopf-type lemma and the nonlocal 
Serrin corner point lemma established in~\cite[Proposition~3.3 and Lemma~4.4]{MR3395749}: see Theorems~\ref{thm:quantitative Hopf} and~\ref{thm:quantitative corner lemma} below.

\begin{prop}\thlabel{Quantitative Maximum Principle}
Let~\(H\subset \R^n\) be an open halfspace and~\(U \) be an open subset of~$H$. Let~$U$, $a\in \ol{H} $ and~$\rho>0$ be such that~\(B_\rho(a)\cap H \subset U\), and set~$a_*:=Q_{\pa H}(a)$. Also, let~\(c\in L^\infty(U)\) be such that
\begin{equation}\label{eq:condition with eigenvalue}
\|c^+\|_{L^\infty(U)} < \lambda_1(B_\rho(a)\cap H).
\end{equation} 

Let~$K\subset H$ be a non-empty bounded open set that is disjoint from~$B_\rho(a)$ and let~$0 \le \sS <\infty$ be such that
\begin{equation}\label{eq:def Sup vecchio M^-1}
\sup_{\genfrac{}{}{0pt}{2}{x\in K}{y\in B_\rho(a)\cap H}}
\vert Q_{\partial H}(x)-y\vert \leq \sS .
\end{equation}

If~$v\in H^s(\R^n)$ is antisymmetric with respect to~$\pa H$ and satisfies
\begin{equation*}
\begin{cases}
(-\Delta)^s v +c v \geq 0 \quad & \text{ in } U  ,\\
v \geq 0 \quad & \text{ in } H ,
\end{cases}
\end{equation*}
in the weak sense, then
we have that
\begin{equation}\label{eq:quantitative maximum principle}
v(x) \geq C \| \delta_{\partial H} v\|_{L^1(K)} \, \de_{\pa H}(x) \, \left( \rho^2 - |x-a|^2 \right)_+^s 
\end{equation}
\begin{equation*}
\begin{cases}
(i) \ \ \text{ for a.e. } x \in H \cap  B_{\rho}(a) , & \quad \text{if } \de_{\pa H}(a)=0 ,
\\
(ii) \ \text{ for a.e. } x \in H \cap B_{\rho}(a) , & \quad \text{if } 0 < \de_{\pa H}(a) \le \rho/2 ,
\\
(iii) \text{ for a.e. } x \in B_{\rho}(a) , & \quad \text{if } \de_{\pa H}(a) > \rho ,
\end{cases}
\end{equation*}
with
\begin{equation*}
C := C(n,s) \frac{\rho^{2s}}{\sS^{n+2s+2}} \Big(1+\rho^{2s}\|c^+\|_{L^\infty(U)} \Big)^{-1}.
\end{equation*}
\end{prop}

\begin{rem}\label{rem:remark on Faber}
{\rm
Notice that~\eqref{eq:condition with eigenvalue} is always satisfied provided that~$\rho$ is small enough. More precisely, the property in~\eqref{eq:Faber-Krahn} and the fact that~$|B_\rho (a) \cap H| \le |B_1|  \rho^n$ give that~\eqref{eq:condition with eigenvalue} is satisfied provided that
\begin{equation*}
\rho \le \frac{ \ka_{n,s}^{\frac{1}{2s}} }{|B_1|^{\frac{1}{n}} \|c^+\|_{L^\infty(U)}^{\frac{1}{2s}}} .
\end{equation*}
}
\end{rem}

\begin{rem}
{\rm
The three cases
in the statement of \thref{Quantitative Maximum Principle} may be gathered together in the following unified statement: {\it \eqref{eq:quantitative maximum principle} holds for a.e.~$x \in H \cap  B_{\rho}(a)$ provided that either~$\de_{\pa H}(a) \in \left[ 0, \rho/2 \right]$ or~$\de_{\pa H}(a) > \rho$.
}

Nevertheless, we prefer to keep the statement of \thref{Quantitative Maximum Principle} with the three cases to clarify their roles in their forthcoming applications.
}
\end{rem}

Before we give the proof of \thref{Quantitative Maximum Principle}, we need the following lemma which says that the pointwise inequality of \thref{prop:newantisymmetricbarrier} can be used in the variational formulation in terms of \(\mathcal E\).

\begin{lem} \thlabel{r4JZ3zkE}
    Let~$\varphi$ be as in~\eqref{eq:def new barrier} and set~$B:=B_\rho(a)$ and~$B_*:=B_\rho( a_* )$. Then, for all \(\xi \in C^\infty_0(B^+\setminus B_\ast)\), there holds \begin{align*}
        \mathcal E (\varphi, \xi ) \leq \frac{n+2s} n \int_{\R^n} x_1 \xi(x) \dd x. 
    \end{align*}
\end{lem}

\begin{proof}
    Let \begin{align*}
    (-\Delta)^s_\varepsilon \varphi(x) := c_{n,s} \int_{\R^n \setminus B_\varepsilon(x) } \frac{\varphi(x)-\varphi(y)}{\vert x-y\vert^{n+2s}} \dd y .
    \end{align*}  Then, by symmetry, \begin{align*}
        \int_{\R^n } \xi(x) (-\Delta)^s_\varepsilon\varphi(x) \dd x &= c_{n,s}\iint_{\R^{2n} \cap \{ \vert x-y\vert \geq \varepsilon \} } \frac{\xi(x)  (\varphi(x)-\varphi(y))}{\vert x-y\vert^{n+2s}} \dd y \dd x \\
        &= \frac{c_{n,s}} 2 \iint_{\R^{2n} \cap \{ \vert x-y\vert \geq \varepsilon \} } \frac{(\xi(x) -\xi(y) )(\varphi(x)-\varphi(y))}{\vert x-y\vert^{n+2s}} \dd y \dd x
    \end{align*} On one hand, since \(\varphi,\xi \in H^s(\R^n)\), we have that \begin{align*}
        \bigg \vert \chf_{\{\vert x-y\vert \geq \varepsilon \} } \frac{(\xi(x) -\xi(y) )(\varphi(x)-\varphi(y))}{\vert x-y\vert^{n+2s}}\bigg \vert \leq \frac{\vert \xi(x) -\xi(y) \vert  \vert \varphi(x)-\varphi(y)\vert }{\vert x-y\vert^{n+2s}} \in L^1(\R^{2n}),
    \end{align*} so the dominated convergence theorem implies that \begin{align*}
        \lim_{\varepsilon \to 0^+ }\frac{c_{n,s}} 2 \iint_{\R^{2n} \cap \{ \vert x-y\vert \geq \varepsilon \} } \frac{(\xi(x) -\xi(y) )(\varphi(x)-\varphi(y))}{\vert x-y\vert^{n+2s}} \dd y \dd x = \mathcal E(\varphi,\xi). 
    \end{align*} On the other hand, since \(\varphi\) is smooth in \(B^+\setminus B_\ast\) and \(\supp \xi\) is compactly contained in \(B^+\setminus B_\ast\), standard estimates for the fractional Laplacian implies \((-\Delta)^s_\varepsilon \varphi \to (-\Delta)^s\varphi \) uniformly in \(\supp \xi\). Hence, \begin{align*}
         \int_{\R^n } \xi(x) (-\Delta)^s_\varepsilon\varphi(x) \dd x  \to \int_{\R^n } \xi(x) (-\Delta)^s\varphi(x) \dd x.
    \end{align*} Thus, \begin{align*}
         \mathcal E(\varphi,\xi) = \int_{\R^n } \xi(x) (-\Delta)^s \varphi(x) \dd x \leq \frac{n+2s} n \int_{\R^n} x_1 \xi(x) \dd x
    \end{align*} by \thref{prop:newantisymmetricbarrier}. 
\end{proof}

We may now give the proof of \thref{Quantitative Maximum Principle}.

\begin{proof}[Proof of \thref{Quantitative Maximum Principle}]
Without loss of generality, we may assume that
$$
H = \R^n_+ = \big\{ x=(x_1,\dots,x_n)\in\R^n \, \text{ : } \, x_1>0 \big\} ,
$$
and hence use the notations~$x_*:= Q_{\pa H }(x) = Q_{\pa \R^n_+ }(x)=x-2x_1e_1$, for all~$x\in\R^n$, and~$A^+ = A \cap \R^n_+$, for all~$A\subset \R^n$.
 \smallskip
 
We first consider case~$(ii)$ (i.e., $0<\de_{\pa H}(a)=a_1\le \rho/2$), which is the most complicated, and take for granted the result in case~$(i)$, whose (simpler) proof is postponed below. Set~$B:=B_\rho(a)$ and~$B_*:=B_\rho( a_* )$.
Let~$\tau \geq 0$ be a constant to be chosen later and set
$$
w := \tau \varphi +  (\chf_K+\chf_{K_*}) v ,
$$
where~$K_*:=Q_{\pa \R^n_+}(K)$ and~$\varphi$ is as in~\eqref{eq:def new barrier}. By \thref{r4JZ3zkE}, for any~$\xi \in C^\infty_0(B^+ \setminus \ol{B_*} )$ with~$\xi \geqslant 0$, we have that 
\begin{equation*}
\begin{split}&
\mathcal E (w, \xi )\\
=\;& \tau \mathcal E (\varphi, \xi )+ \mathcal E (\chf_K v , \xi )+ \mathcal E (\chf_{K_*} v , \xi ) 
\\
\leq\;& \frac{n+2s}{n} \,  \tau \int_{ \supp(\xi) } x_1 \xi(x) \dd x - c_{n,s} \left[ 
\iint_{ \supp(\xi)\times K} \frac{\xi(x) v(y)}{\vert x-y\vert^{n+2s}} \dd x  \dd y+
\iint_{ \supp(\xi)\times K_*} \frac{\xi(x) v(y)}{\vert x-y\vert^{n+2s}} \dd x  \dd y\right]
\\
=\;& \int_{\supp(\xi) } \bigg [ \frac{n+2s}{n} \, \tau x_1 - c_{n,s}  \int_{K}  \bigg ( \frac 1{\vert x-y\vert^{n+2s}} - \frac 1 {\vert x_* - y \vert^{n+2s}} \bigg ) v(y) \dd y \bigg ] \xi(x) \dd x ,
\end{split}
\end{equation*}
where~$\supp(\xi) $ denotes the support of~$\xi$. Note that \begin{align*}
    \bigg \vert \iint_{ \supp(\xi)\times K} \frac{\xi(x) v(y)}{\vert x-y\vert^{n+2s}} \dd x  \dd y\bigg \vert <+\infty
\end{align*} since \(\dist (\supp \xi , K)>0\) (using that \(K\) and \(B\) are disjoint and \(\supp \xi\) is compactly contained in \(B\)), so \begin{align*}
    \bigg \vert \iint_{ \supp(\xi)\times K} \frac{\xi(x) v(y)}{\vert x-y\vert^{n+2s}} \dd x  \dd y \bigg \vert &\leq \dist (\supp \xi , K)^{-n-2s} \| \xi\|_{L^1(\R^n)} \| v\|_{L^1(K)} \\
    &\leq \dist (\supp \xi , K)^{-n-2s} \vert K \vert^{1/2} \| \xi\|_{L^1(\R^n)} \| v\|_{L^2(\R^n)}. 
\end{align*} We obtain a similar estimate for the integral involving \(K_\ast\).

Recalling~\eqref{eq:def Sup vecchio M^-1} and noting that, for all~\(x\in B^+\) and~\(y\in K\),
\begin{equation*}
\begin{split}
\frac 1{\vert x-y\vert^{n+2s}} - \frac 1 {\vert x_*- y \vert^{n+2s}}  
& = \frac{n+2s}2 \int_{\vert x- y \vert^2}^{\vert x_*- y \vert^2} t^{-\frac{n+2s+2}2} \dd t
\\
& \geq \frac{n+2s}2  \;\frac{ \vert x_* - y \vert^2 - \vert x- y \vert^2 }{ \vert x_*- y \vert^{n+2s+2}}
\\
&= 2(n+2s)  \frac{x_1y_1}{\vert x_* -y \vert^{n+2s+2} } 
\\
&\geq \frac{2(n+2s)}{\sS^{n+2s+2}} x_1y_1 ,
\end{split}
\end{equation*}
we obtain that
\begin{eqnarray*}
\mathcal E(w,\xi) &\leq& \left(
\frac{n+2s}n\,\tau -\frac{2c_{n,s}(n+2s)}{\sS^{n+2s+2}}\,\|y_1v\|_{L^1(K)} \right)  \int_{\supp(\xi) }   x_1 \xi(x) \dd x 
\\&=&
C \left( \tau - \widetilde C \,  \frac{ \| y_1 v \|_{L^1(K)}}{\sS^{n+2s+2}} \right)  \int_{\supp(\xi) }   x_1 \xi(x) \dd x ,
\end{eqnarray*}
where~\(C\) and~\(\widetilde C\) depend only on~\(n\) and~\(s\).

Hence, using that~$w = \tau \varphi = \tau x_1 \psi_B \leq \ga_{n,s} \rho^{2s} \tau x_1$ in~$B^+ \setminus \ol{B_*} \supset \supp(\xi) $, we have that
\begin{equation*}
\mathcal E(w,\xi) + \int_{\supp(\xi) } c(x) w(x) \xi(x) \dd x  
\leq
C  \left[ \tau \left( 1  + \rho^{2s}\|c^+\|_{L^\infty(U)} \right)  - \widetilde C \, \frac{  \| y_1 v \|_{L^1(K)} }{\sS^{n+2s+2}} \right] \int_{\supp(\xi)}   x_1 \xi(x) \dd x ,
\end{equation*}
where~$C$ and~$\widetilde C$ may have changed from the previous formula but still depend only on~$n$ and~$s$. 

We thus get that
\begin{equation*}
(-\Delta)^s w + cw \leq 0 \quad \text{ in } B^+ \setminus \ol{B_*},\end{equation*}
provided that \begin{equation*}
\tau \le \frac{ \widetilde C\, \| y_1 v \|_{L^1(K)}}{2\sS^{n+2s+2}} \left( 1  + \rho^{2s}\|c^+\|_{L^\infty(U)} \right)^{-1}
. 
\end{equation*}

We now claim that
\begin{equation}\label{deiwty0987654}
w\le v \quad {\mbox{ in }}\R^n_+\setminus\big( B^+ \setminus \ol{B_*}\big).
\end{equation}
To this end, we notice that, in~$\R^n_+ \setminus B^+$, we have 
that~$w = \chf_K v \le  v$. Hence, to complete the proof of~\eqref{deiwty0987654}
it remains to check that
\begin{equation}\label{deiwty098765400}
w \le v\quad {\mbox{ in }}B_*^+,\end{equation}
provided that~$\tau$ is small enough. In order to prove this, we set~$\widetilde{B}:=B_{\sqrt{\rho^2 - a_1^2}}\left( \frac{a+a_*}{2} \right)$ and notice that~$B \cap B_* \subset \widetilde{B} \subset B \cup B_*$.
Thus, an application of item~$(i)$ in~$\widetilde{B}$ gives that 
\begin{equation*}
v \ge \widehat{C} \, \frac{ (\rho^{2} - a_1^2)^s}{\sS^{n+2s+2}} \Big(1+(\rho^{2}-a_1^2)^s \|c^+\|_{L^\infty(U)} \Big)^{-1} \,\| y_1 v \|_{L^1(K)}\, x_1 \psi_{\widetilde{B}}(x)  \quad \text{ in } \widetilde{B}^+ ,
\end{equation*}
where~$\widehat{C}$ is a constant only depending on~$n$ and~$s$.
Hence, since~$ 0 < a_1 \le \rho/2$,
\begin{equation}\label{diweoghuyoiuypoiu98765}
v \ge \left( \frac{3}{4} \right)^s  \frac{\widehat{C}\, \rho^{2s}}{\sS^{n+2s+2}} \Big( 1 + \rho^{2s} \|c^+\|_{L^\infty(U)} \Big)^{-1} \,\| y_1 v \|_{L^1(K)}\, x_1 \psi_{\widetilde{B}}(x)  \quad \text{ in } \widetilde{B}^+ ,
\end{equation}
up to renaming~$\widehat{C}$.

Moreover, we claim that
\begin{equation}\label{topoeugtf}
\psi_{\widetilde{B}} \ge \frac{1}{2} \left( \psi_{B} + \psi_{B_*} \right) \quad \text{ in } B_*^+ .
\end{equation}
To check this, we observe that, for all~$t\in\left(0,\frac12\right)$,
\begin{equation}\label{pluto}
2(1-t)^s-(1-2t)^s \ge 1.\end{equation}
Also, if~$x\in B_*^+$, 
we have that~$\rho^2>|x-a_*|^2=|x-a|^2+4x_1a_1$. 
Hence, setting~$t:=\frac{2a_1x_1}{\rho^2-|x-a|^2}$, we see that~$t\in\left(0,\frac12\right)$ and thus we obtain from~\eqref{pluto} that
\begin{eqnarray*}
1\le 2\left(1-\frac{2a_1x_1}{\rho^2-|x-a|^2}\right)^s-\left(1-\frac{4a_1x_1}{\rho^2-|x-a|^2}\right)^s, 
\end{eqnarray*} 
that is
\begin{equation}\label{oewyjiythjvdnsjfe0w9876PPPP}\begin{split}
\big(\rho^2-|x-a|^2\big)^s\le\;& 2\Big( \rho^2-|x-a|^2 
- 2a_1x_1\Big)^s-\Big( \rho^2-|x-a|^2-4a_1x_1\Big)^s \\
=\;&2\left( \rho^2-a_1^2-\left| x-\frac{a+a_*}2\right|^2
\right)^s-\Big( \rho^2-|x-a_*|^2\Big)^s.
\end{split}\end{equation} 

Since~$x\in B_*^+$, we have that~$x\in B\cap B_*^+\subset \widetilde{B}$,
and therefore~\eqref{oewyjiythjvdnsjfe0w9876PPPP} gives that~$ \psi_{B}(x)\le 2\psi_{\widetilde{B}}(x)- \psi_{B_*}(x) $,
which is the desired result in~\eqref{topoeugtf}.

Putting together~\eqref{diweoghuyoiuypoiu98765} and~\eqref{topoeugtf}, 
we conclude that
\begin{equation*}
v \ge \left( \frac{3}{4} \right)^s \frac{ \widehat{C} \, \rho^{2s}}{2\sS^{n+2s+2}} \Big( 1 + \rho^{2s} \|c^+\|_{L^\infty(U)} \Big)^{-1} \,\| y_1 v \|_{L^1(K)}\,\varphi  \quad \text{ in } {B}_*^+ .
\end{equation*}
As a consequence, if
$$  \tau \le \left( \frac{3}{4} \right)^s \frac{\widehat{C}\,\rho^{2s}\,\| y_1 v \|_{L^1(K)}}{2\sS^{n+2s+2}} \Big(1+\rho^{2s}\|c^+\|_{L^\infty(U)} \Big)^{-1}, $$
we obtain that~$
w= \tau \varphi \le v$ in~$ B_*^+$.
This proves~\eqref{deiwty098765400}, and so~\eqref{deiwty0987654}
is established.

Gathering all these pieces of information, we conclude that, setting
\begin{equation*}
\tau :=  \min \left\lbrace \widetilde C , \left( \frac{3}{4} \right)^s \widehat{C}\rho^{2s} \, \right\rbrace  \, \frac{  \| y_1 v \|_{L^1(K)}}{2\sS^{n+2s+2}} \left( 1  + \rho^{2s}\|c^+\|_{L^\infty(U)} \right)^{-1} ,
\end{equation*} 
we have that
\begin{equation*}
\begin{cases}
(-\Delta)^s w + cw \leq 0 \quad \text{ in } B^+ \setminus \ol{B_*},
\\
v \ge w \quad \text{ in } \R^n_+ \setminus \left( B^+ \setminus \ol{B_*} \right) .
\end{cases}
\end{equation*}
Hence, recalling also~\eqref{eq:condition with eigenvalue}, the comparison principle in~\cite[Proposition~3.1]{MR3395749} 
(see also~\cite[Proposition~2.3]{zbMATH07106126})
gives that,
in~\( B^+ \setminus \ol{B_*} \),
\begin{equation*}
v(x) \geq w(x)  =   \min \left\lbrace \widetilde C , \left( \frac{3}{4} \right)^s \widehat{C} \, \rho^{2s}\,\right\rbrace \, \frac{ \| y_1 v \|_{L^1(K)} }{2\sS^{n+2s+2}} 
\,\left( 1+\rho^{2s}\|c^+\|_{L^\infty(U)} \right)^{-1}
\, x_1 \, \psi_B (x) ,
\end{equation*}
from which the desired result (in case~$(ii)$) follows.
\smallskip 

In case~$(i)$, that is the case where~$\de_{\pa H}(a) = a_1=0$ (i.e., $a \in \pa H$), it is easy to check that the desired result follows by repeating the above argument, but simply with~$\varphi := x_1 \psi_B$ in place of the barrier in~\eqref{eq:def new barrier} and~$B^+$ in place of~$B^+ \setminus \ol{B_*}$. The argument significantly simplifies, as, in this case, we do not need to check the inequality~$w \le v$ in~$B_*$, as given by~\eqref{deiwty098765400}, before applying~\cite[Proposition~3.1]{MR3395749}.
\smallskip

In case~$(iii)$, i.e., the case where~$\de_{\pa H}(a) = a_1 > \rho$, the desired result follows by repeating the same argument in item~$(ii)$. Notice that in this case, we have that~$B \setminus \ol{B}_* = B = B^+$, and hence (as in case~$(i)$) the argument significantly simplifies, as we do not need to check the inequality~$w \le v$ in~$B_*$ before applying~\cite[Proposition~3.1]{MR3395749}.
\end{proof}

\begin{figure}[htbp]
    \centering
    \includegraphics[width=0.7\textwidth]{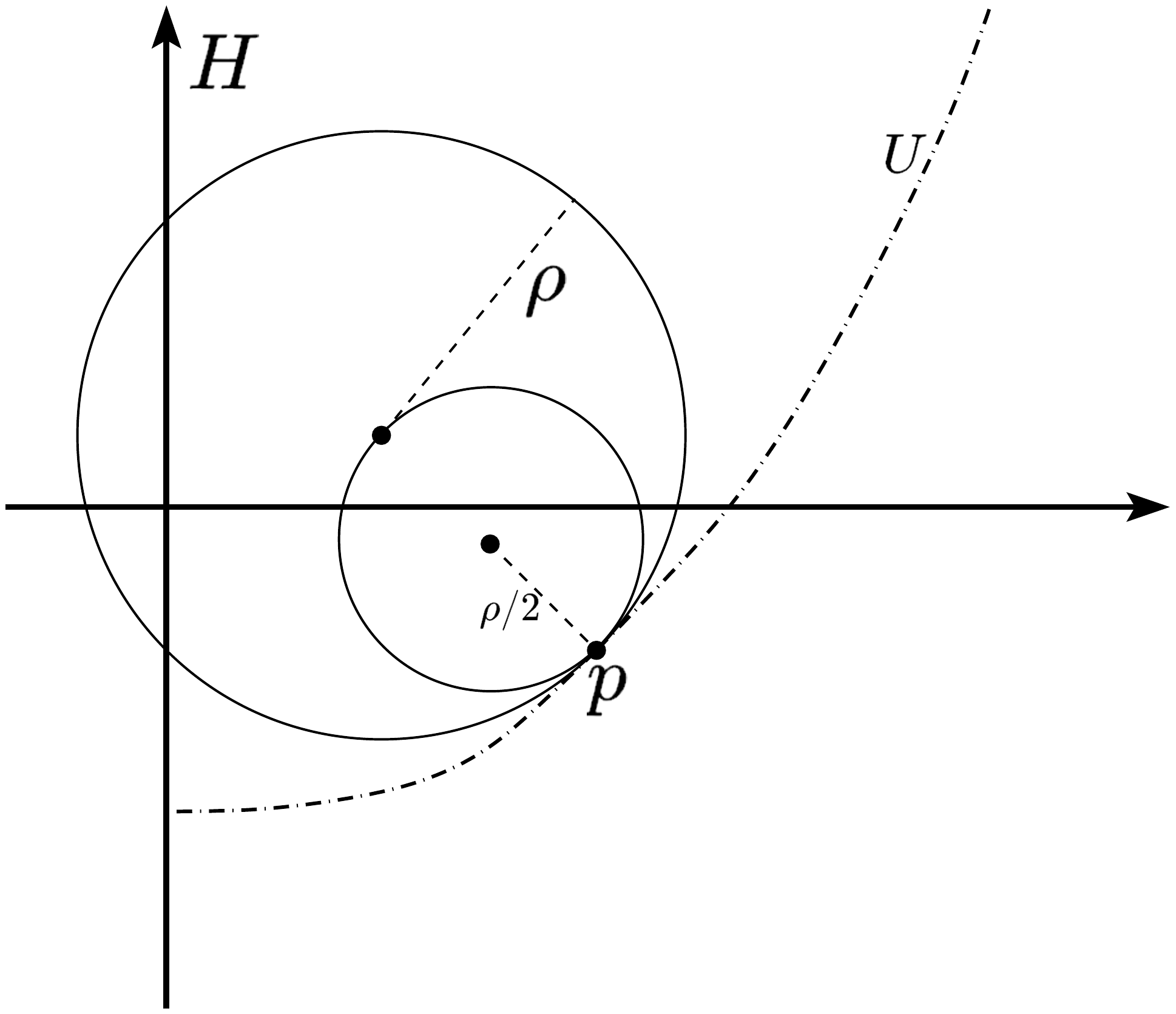}
    \caption{The geometry described in footnote~\ref{fofig:mylabel}.}
    \label{fig:mylabel}
\end{figure}

\begin{thm}[Quantitative nonlocal Hopf Lemma]\thlabel{thm:quantitative Hopf}
Under the assumptions of \thref{Quantitative Maximum Principle},
if~$a\in H$, $v \in C (\ol{B_\rho(a)} \setminus B_\rho(a_*) )$,
and~$v(p)=0$ for some~$p=(p_1,p_2,\dots,p_n) \in  \pa B_\rho(a)\cap H$ with
\begin{equation}\label{AD:Hy82}
	\de_{\pa H} (p) \ge \de_{\pa H} (a),
\end{equation}
then we have that
\begin{equation}\label{eq:Hopf}
\liminf_{t\to 0^+} \frac{ v(p-t \nu(p)) - v(p) }{ t^s \,  \de_{\pa H}(p-t \nu(p)) } \ge \ol{C} \, \| \delta_{\partial H} v\|_{L^1(K)},
\end{equation}
with
\begin{equation}\label{eq:def ol C}
\ol{C} := C(n,s) \frac{\rho^{3s}}{\sS^{n+2s+2}} \Big(1+\rho^{2s}\|c^+\|_{L^\infty(U)} \Big)^{-1}.
\end{equation}
Here, $\nu(p)$ is the exterior unit normal of~$\pa B_\rho(a)$ at~$p$.
\end{thm}

\begin{proof}
The aim is to use the estimate
in~\eqref{eq:quantitative maximum principle} of Proposition~\ref{Quantitative Maximum Principle}.

More precisely, 
\begin{itemize}
\item if~$0 < \de_{\pa H}(p - \rho \, \nu(p)) \le \rho/2$, we use item~$(ii)$
with~$a$ replaced by~$p - \rho \, \nu(p)$
\item if~$\rho/2 < \de_{\pa H}(p - \rho \, \nu(p)) \le \rho$, 
we use\footnote{We stress that, to use item~$(iii)$ here, one needs to check that\label{fofig:mylabel}
$$ B_{\rho/2}\left(p - \frac\rho2 \, \nu(p)\right)\subseteq H.$$
To validate this, one can suppose, without loss of generality, that~$H=\{x_1>0\}$. Then, by~\eqref{AD:Hy82},
we know that~$\nu_1(p)\ge0$ and therefore, for all unit vectors~$e$,
$$ \left(p - \frac\rho2 \, \nu(p)+\frac\rho2\,e\right)\cdot e_1\ge
\left(p - \rho \nu(p)\right)\cdot e_1+\frac\rho2 \, \nu_1(p)-\frac\rho2\ge
\left(p - \rho \nu(p)\right)\cdot e_1-\frac\rho2>0,$$
which establishes the desired inclusion.} item~$(iii)$
with~$a$ replaced by~$p - \frac\rho2 \, \nu(p)$ and~$\rho$ replaced by~$\rho/2$,
\item if~$\de_{\pa H}(p - \rho \, \nu(p)) > \rho$, we
use item~$(iii)$
with~$a$ replaced by~$p - \rho \, \nu(p)$.
\end{itemize}
Accordingly, using~\eqref{eq:quantitative maximum principle}
with~$x:=p - t \nu(p)$, up to renaming constants when~$\rho$ is replaced by~$\frac\rho2$, we have that
\begin{eqnarray*}
v\big(p - t \nu(p)\big) &\geq& C \| \delta_{\partial H} v\|_{L^1(K)} \, \de_{\pa H}\big(p - t \nu(p)\big) \, \Big( \rho^2 - |p - t \nu(p)-(p - \rho\nu(p))|^2 \Big)_+^s \\
&=& C \| \delta_{\partial H} v\|_{L^1(K)} \, \de_{\pa H}\big(p - t \nu(p)\big) \, 
t^s(2\rho-t)^s \\
&\ge&C \| \delta_{\partial H} v\|_{L^1(K)} \, \de_{\pa H}\big(p - t \nu(p)\big) \, 
t^s\rho^s.
\end{eqnarray*}
Taking the limit as~$t\to0^+$, 
we obtain the desired result in~\eqref{eq:Hopf} and~\eqref{eq:def ol C}.
\end{proof}

\begin{thm}[Quantitative nonlocal Serrin corner point lemma]\thlabel{thm:quantitative corner lemma}
Let~$\Om \in \R^n$ be an open bounded set with~$C^2$ boundary such that the origin~$0\in \pa\Om$. Assume that the hyperplane~$\pa H= \left\lbrace x_1 = 0 \right\rbrace$ is orthogonal to~$\pa\Om$ at~$0$.
Let the assumptions in \thref{Quantitative Maximum Principle} be satisfied with~$U := \Om \cap H $, $a:=(0,\rho,0,\dots,0)\in \pa H$ and~$0 \in \pa U \cap \pa( B_\rho(a)\cap H)$.

Then, setting~$\eta:=(1,1,0\dots,0)$, we have that
\begin{equation}\label{eq:Serrin corner lemma}
\frac{v(t \eta)}{t^{1+s}} \ge \ol{C} \, \| \delta_{\partial H} v\|_{L^1(K)}  \quad \text{ for } 0< t < 
\frac\rho2,
\end{equation}
with~$\ol{C}$ in the same form as~\eqref{eq:def ol C}.
\end{thm}

\begin{proof}
Using~\eqref{eq:quantitative maximum principle} (item~$(i)$) with~$x= t \eta$ (for~$0< t <\rho/2$), we see that
\begin{eqnarray*}
v(t\eta) &\geq& C \| \delta_{\partial H} v\|_{L^1(K)} \, \de_{\pa H}(t\eta) \, \left( \rho^2 - |t\eta-a|^2 \right)_+^s \\
&=& C \| \delta_{\partial H} v\|_{L^1(K)} \, 2^s t^{1+s}( \rho-t)^s\\
&\ge& C \| \delta_{\partial H} v\|_{L^1(K)} \, t^{1+s} \rho^s.
\end{eqnarray*}
The desired result follows. 
\end{proof}

\begin{rem}
{\rm
In the particular case where~$p$ is far from~$\pa H$ (say~$\de_{\pa H}(p) > 2 \rho$), \thref{thm:quantitative Hopf} is a slight improvement of a result that could essentially be deduced from~\cite[Lemma~4.1]{MR4577340}. The key novelty and huge improvement provided by \thref{Quantitative Maximum Principle} is that \thref{thm:quantitative Hopf} remains valid regardless of the closeness of~$p$ to~$\pa H$.
Notice that~\eqref{eq:Hopf} becomes worse as~$p$ becomes closer to~$\pa H$, and, heuristically, becomes qualitatively similar to~\eqref{eq:Serrin corner lemma} for~$\de_{\pa H} (p) \to 0$.
}
\end{rem}

\section{Quantitative moving planes method and proof of Theorem~\ref{thm:Main Theorem}}\label{sec:main result}

Given~\(e\in \Sph^{n-1}\), \(\mu \in \R\), and~$A \subset \R^n$, we will use the following definitions, as in~\cite{DPTVParallelStability}: 
$$
\pi_\mu :=\{ x\in \R^n \text{ : } x\cdot e = \mu \} , \quad  \text{ i.e., the hyperplane orthogonal to }e \text{ and containing } \mu e ,
$$
$$
H_\mu :=\{x\in \R^n  \text{ : } x\cdot e>\mu \} , \quad \text{ i.e., the right-hand half space with respect to } \pi_\mu  ,
$$
$$
A_\mu := A \cap H_\mu , \quad \text{ i.e., the portion of }A \text{ on the right-hand side of } \pi_\mu ,
$$
$$
x_\mu' := x-2(x\cdot e -\mu) e ,  \quad  \text{ i.e., the reflection of } x \text{ across } \pi_\mu ,
$$
$$
A_\mu' := \{x\in \R^n \text{ : } x'_\mu\in A_\mu \} , \quad \text{ i.e., the reflection of } A_\mu \text{ across } \pi_\mu .
$$
It is clear from the above definitions that
$$
H_\mu' =\{x\in \R^n  \text{ : } x\cdot e<\mu \} , \quad \text{ i.e., the left-hand half space with respect to } \pi_\mu .
$$

Let~$\Om$ be smooth and let~$u$ be a solution of~\eqref{eq:Dirichlet problem}. Given~$\mu \in \R$ and setting
\begin{equation*}
v_\mu(x) := u(x) - u( x_\mu' ), \qquad \text{ for }  x\in \R^n ,
\end{equation*}
we have that
\begin{equation*}
(-\Delta)^s v_\mu (x) = f(u(x)) - f(u(x_\mu')) = -c_\mu(x) v_\mu (x) 
\qquad \text{ in } \Omega_\mu' ,
\end{equation*}
where
\begin{equation*}
c_\mu(x) := 
\begin{cases}
-\displaystyle \frac{f(u(x)) - f(u(x_\mu'))}{u(x) - u(x_\mu')} \quad & \text{ if } u(x) \neq u(x_\mu') ,
\\
0 \quad & \text{ if } u(x) = u(x_\mu') .
\end{cases}
\end{equation*}

Thus, $v_\mu$ is an antisymmetric function satisfying
\begin{equation*}
\begin{cases}
(-\Delta)^s v_\mu  +c_\mu v_\mu = 0 \quad &\text{ in } \Omega_\mu ' ,\\
v_\mu = u \quad &\text{ in } (\Omega\cap H_\mu') \setminus \Omega_\mu' ,\\
v_\mu = 0 \quad &\text{ in } H_\mu' \setminus \Omega ,
\end{cases}
\end{equation*}
with~\(c_\mu \in L^\infty (\Omega_\mu ')\). Moreover, we have that
\begin{equation}\label{eq: upper bound infinity norm of c}
\| c_\mu \|_{L^\infty(\Omega_\mu ')} \leq [f]_{C^{0,1}([0,\| u\|_{L^\infty(\Omega)}])}.
\end{equation}

Given a direction~$e \in \Sph^{n-1}$ and defining~$\La=\La_e:= \sup \left\lbrace  x \cdot e \, \text{ : } \, x \in \Om \right\rbrace$, we consider the \emph{critical value}
\begin{equation*}
\la=\la_e:= \inf \Big\{ \ul{\mu} \in \R \, \text{ : } \, \Omega_\mu ' \subset \Om \text{ for all } \mu \in \big( \ul{\mu} , \La \big)  \Big\} ,
\end{equation*}
for which, as usual in the method of the moving planes, either
\begin{itemize}
\item[(i)] $\pa \Om_\la'$ is internally tangent to~$\pa \Om$ at a point~$p \notin \pi_\la$,
\item[(ii)] or the critical plane~$\pi_\la$ is orthogonal to~$\pa\Om$ at a point~$p \in \pi_\la \cap \pa \Om$.
\end{itemize}
\smallskip

We recall that~$v_\mu \ge 0 $ in~$\Om_\mu'$ for all~$\mu \in \left[ \la, \La \right]$: see, e.g., \cite[Lemma~4.1]{DPTVParallelStability}.

\begin{lem} \thlabel{lem:symmetric difference}
Let~$\Om$ be an open bounded set of class~$C^2$ satisfying the uniform interior sphere condition with radius~$r_\Om>0$. 
Let~$f \in C^{0,1}_{\mathrm{loc}}(\R)$ be such that~\(f(0)\geqslant 0\) and define~$R$ as in~\eqref{eq:def R}. Let~$u$ be a weak solution of~\eqref{eq:Dirichlet problem} satisfying~\eqref{eq:higher boundary regularity}.

Then, for any~\(e\in \Sph^{n-1}\), we have that
\begin{equation} \label{eq:lemma symmetric difference}
\int_{(\Omega\cap H'_\lambda) \setminus \Omega_\lambda'} \delta_{\pi_\lambda}(x) u (x) \dd x \leq C  [ \pa_\nu^s u]_{\partial \Om} ,
\end{equation}
where
\begin{equation}\label{eq:constant in lemma symmetric}
C := C(n,s)  \left( \frac{\diam \Om}{R} \right)^{3s} (\diam \Omega)^{n+2-s} . 
\end{equation}
\end{lem}
\begin{proof} 
We use the method of moving planes and, without loss of generality, we take~$e:=-e_1$ and~$\la =0$.

$(i)$ Assume that we are in the first case, that is, the boundary of~$\Om'_\la$ is internally tangent to~$\pa \Om$ at some point~\(p\in (\partial \Om \cap \partial \Om_\lambda') \setminus \{x_1=0\}\).
Notice that, by definition of~$r_\Om$, we have that~$p - \nu(p) \, r_\Om \in H_\la'$.
We now use \thref{thm:quantitative Hopf} with~$H:=H_\la'$, $U:=\Om'_\la$, \(K := (\Om \cap H_\la') \setminus \Om_\la'\), $\rho:=R$ and~$a:= p - R \, \nu(p)$, and we obtain from~\eqref{eq:Hopf} that
\begin{equation}\label{eq:step in lemma}
\int_{(\Omega \cap H_\lambda') \setminus \Omega_\lambda '} y_1 u(y) \dd y
=
\int_{(\Omega \cap H_\lambda') \setminus \Omega_\lambda '} y_1 v_\lambda (y) \dd y 
\le
C \, \frac{ \pa_\nu^s v_\lambda(p)}{p_1} .
\end{equation}
As noticed in Remark~\ref{rem:remark on Faber}, condition~\eqref{eq:condition with eigenvalue}
is satisfied in light of~\eqref{eq:def R} and~\eqref{eq:Faber-Krahn}.
Also notice that, since, for any~$x\in(\Omega\cap H_\lambda') \setminus \Omega_\lambda'$, the point~\(x_\lambda'\) belongs to the reflection of~$\Om$ across~$\{x_1=0\}$, we have that
\eqref{eq:def Sup vecchio M^-1} holds true with~$\sS:= \diam \Om.$
Hence, \thref{Quantitative Maximum Principle}, \eqref{eq: upper bound infinity norm of c} and~\eqref{eq:def R} give that the constant in~\eqref{eq:step in lemma} is as in~\eqref{eq:constant in lemma symmetric}.
Noting that
\begin{equation*}
\frac{ \pa_\nu^s v_\lambda(p)}{p_1} = \frac{ \pa_\nu^s u(p) - \pa_\nu^s u(p_\lambda') }{p_1} = 2 \, \frac{( \pa_\nu^s u(p_\lambda') - \pa_\nu^s u(p) )}{\vert (p_\lambda')_1 - p_1 \vert} \le 2 \, [ \pa_\nu^s u]_{\partial \Om},
\end{equation*}
we obtain~\eqref{eq:lemma symmetric difference} from~\eqref{eq:step in lemma}. 

$(ii)$ In the second case, we can assume (without loss of generality) that~$ 0 = p \in \partial \Om \cap \{x_1=0\}$ and the exterior normal of~$\pa \Om $ at~$ 0 $ agrees with~$e_2$ (which is contained in~\(\{x_1=0\}\)).
We exploit \thref{thm:quantitative corner lemma} with~$H:=H_\la'$, $U:=\Omega'_\lambda$,
$K := (\Om \cap H_\la') \setminus \Om_\la'$, $\rho:=R$ and~$a:= R e_2$,
thus obtaining from~\eqref{eq:Serrin corner lemma} that
\begin{equation}\label{eq:preperstimauno}
\int_{(\Omega\cap H_\lambda') \setminus \Omega_\lambda'} x_1 u (x) \dd x \le C \, \frac{ | v_\lambda( t \eta ) | }{t^{1+s}}
\end{equation}
with \(t\in (0, R/2)\), $C$ in the same form as in~\eqref{eq:constant in lemma symmetric}, and~$\eta := (1,1,0,\dots,0)$. 

Setting~$\de(x)$ to be a smooth function in~$\Om$ that coincides with~$\mathrm{dist}(x, \R^n \setminus \Om)$ in a neighborhood of~$\pa\Om$,
the condition in~\eqref{eq:higher boundary regularity} gives that
$$
u(x) = \de^s(x) \psi (x) \quad \text{ with }\, \psi := \frac{u}{\de^s} \in C^1(\ol{\Om}).
$$
Then, setting~$\eta_* := (-1, 1, 0 \dots, 0)$, we compute that
\begin{equation}\label{eq:preTaylor}
v_\lambda( t \eta ) = u (t \eta) - u (t \eta_* ) = \psi(t \eta) \left[ \de^s(t\eta) - \de^s(t \eta_*) \right] + \de^s (t  \eta_* ) \left[ \psi(t \eta) - \psi(t \eta_* ) \right] .
\end{equation}
As in~\cite[Proof of Lemma~4.3]{MR3395749}, we exploit classical properties of the distance function (see e.g.~\cite{MR1814364}), namely that~$\de(0)=0$, $\na \de (0)= e_2$ and the fact that~$\na^2 \de (0)$ is diagonal, to obtain the Taylor expansions
$$
\de^s(t \eta) = t^s \left( 1+ \frac{s}{2} \langle \na^2 \de(0) \eta, \eta \rangle \, t + o(t) \right) 
$$
and
$$
\de^s(t \eta_* ) = t^s \left( 1+ \frac{s}{2} \langle \na^2 \de(0) \eta_* , \eta_* \rangle \, t + o(t) \right).
$$
As a consequence,
\begin{equation*}
\de^s(t\eta) - \de^s(t \eta_* ) = o(t^{1+s}) .
\end{equation*}
By the continuity of~$\psi$ over~$\ol{\Om}$, this guarantees that
\begin{equation}\label{dewghokjhg09876543w-desfr-34t}
\lim_{t\to 0^+} \frac{\psi(t \eta) \left[ \de^s(t\eta) - \de^s(t \eta_* ) \right]}{t^{1+s}} = 0.
\end{equation}

We now estimate the last summand 
in~\eqref{eq:preTaylor}.
To this aim, a first-order Taylor expansion (for which we use~\eqref{eq:higher boundary regularity}) gives that
$$
\psi (t \eta) = \psi (0) + t \langle \na \psi(0), \eta \rangle + o(t) 
\qquad{\mbox{and}}\qquad
\psi (t \eta_* ) = \psi (0) + t \langle \na \psi(0), \eta_* \rangle + o(t) ,
$$
and therefore
$$
\psi (t \eta) - \psi (t \eta_* ) = 2 t \langle \na \psi(0), e_1 \rangle + o(t) .
$$

Hence, \eqref{eq:preTaylor} and~\eqref{dewghokjhg09876543w-desfr-34t}
give that
$$
v_\lambda( t \eta ) = 2 t^{1+s} \langle \na \psi (0), e_1 \rangle + o(t^{1+s}).
$$
Recalling that~$\psi = u/\de^s$, the definitions in~\eqref{eq:def fracnormal derivative} and~\eqref{eq:def Lipschitz seminorm}, and that~$e_1$ is orthogonal to~$e_2 = \nu (0)$,
we obtain that
$$
\lim_{t \to 0^+} \frac{ | v_\lambda( t \eta ) | }{ t^{1+s} } \le
2 \big|\langle \na \psi (0), e_1 \rangle\big|\le
 2  [\pa_\nu^s u]_{\pa\Om}.
$$
This and~\eqref{eq:preperstimauno} lead to
\begin{equation*}
\int_{(\Omega\cap H_\lambda') \setminus \Omega_\lambda'} x_1 u (x) \dd x  \leqslant C [\pa_\nu^s u]_{\partial \Om} ,
\end{equation*}
and~\eqref{eq:lemma symmetric difference} follows in this case as well. 
\end{proof}

\begin{prop}\thlabel{prop:vecchia prop 4.4}
Let~\(\Omega\) be an open bounded set of class~$C^2$ satisfying the uniform interior sphere condition with radius~\(r_\Omega >0\). Let~$f \in C^{0,1}_{\mathrm{loc}}(\R)$ be such that~\(f(0)\geqslant 0\), and define~$R$ as in~\eqref{eq:def R}. Let~$u$ be a weak solution of~\eqref{eq:Dirichlet problem} satisfying~\eqref{eq:higher boundary regularity}.

For~$e\in \Sph^{n-1}$, denoting with~$\Om'$ the reflection of~$\Om$ with respect to the critical hyperplane~$\pi_\la$, we have that
\begin{equation}\label{eq:vecchia prop 4.4}
\vert \Omega \triangle \Omega' \vert \leqslant C_\star [\pa_\nu^s u]_{\pa\Om}^{\frac1{s+2}}, 
\end{equation}
where
\begin{equation}\label{eq:C star}
C_\star := C(n,s) \left[1+
\frac{ 1 }{ r_\Om^{1-s}  \left( f(0)+\|u\|_{L_s(\R^n)} \right) } \right]
\left( \frac{\diam \Omega}{R} \right)^{2n+2+6s} 
R^{n+2-s} r_\Om^{n-1}.
\end{equation}
\end{prop}
\begin{proof}
The proof runs as the proof of Proposition~4.4 in~\cite{DPTVParallelStability}, but using \thref{lem:symmetric difference} instead of~\cite[Lemma~4.3]{DPTVParallelStability}.
\end{proof}

Notice that, if the Serrin-type overdetermined condition~\eqref{eq:overdetermination} is in force then~\([ \pa_\nu^s u]_{\partial \Om}=0\), and hence \thref{prop:vecchia prop 4.4} gives that, for every direction, $\Om$ must be symmetric with respect to the critical hyperplane. This implies that~$\Om$ must be a ball, and hence recovers the main result of~\cite{MR3395749}.

Moreover,
\thref{prop:vecchia prop 4.4} allows us to obtain \thref{thm:Main Theorem} by exploiting tools of the quantitative method of the moving planes that have been developed in~\cite{MR3836150} (see also~\cite{MR4577340, RoleAntisym2022, DPTVParallelStability}).

\begin{proof}[Proof \thref{thm:Main Theorem}]
The proof can be completed by combining \thref{prop:vecchia prop 4.4} with tools already used in~\cite{MR3836150, DPTVParallelStability}.
We can assume that, for every~$i= 1, \dots, n$, the critical plane~$\pi_{e_i}$ with respect to the coordinate direction~$e_i$ coincides with~$\left\lbrace x_i = 0 \right\rbrace $ . Given~$e \in S^{n-1}$, denote by~$\lambda_e$ the critical value associated with~$e$.
Reasoning as in the proof of~\cite[Proposition~4.5]{DPTVParallelStability}, but using \thref{prop:vecchia prop 4.4} instead of~\cite[Proposition~4.4]{DPTVParallelStability}, we get that, if
\begin{equation}\label{eq:assumption for lambda estimate}
[\pa_\nu^su]^{\frac 1{s+2} }_{\partial G} \leq \frac {\vert \Omega \vert } {n C_\star }  ,
\qquad \text{ with } C_\star \text{ as in~\eqref{eq:C star}},  
\end{equation}
then
\begin{equation}\label{eq:lambda estimate frac normal der}
\vert \lambda_e \vert \leqslant C [\pa_\nu^s u]_{\partial \Om}^{\frac 1 {s+2} }
\quad \text{ for all } e\in \Sph^{n-1} ,
\end{equation}
with
\begin{equation*}
C := C(n,s) \left[1+
\frac{ 1 }{ r_\Om^{1-s}  \left( f(0)+\|u\|_{L_s(\R^n)} \right) } \right]
\left( \frac{\diam \Omega}{R} \right)^{2n+2+6s} 
R^{n+2-s} \,
\left( \frac{r_\Om^{n-1} \diam \Om}{|\Om|} \right)
.
\end{equation*}
\thref{thm:Main Theorem} follows by reasoning as in the proof of~\cite[Theorem~1.2]{MR3836150} but using~\eqref{eq:assumption for lambda estimate} and~\eqref{eq:lambda estimate frac normal der} instead of~\cite[Lemma~4.1]{MR3836150}. Notice that, recalling the definition of~$r_\Om$ and~\eqref{eq:def R}, the measure of~$\Omega$ can be easily estimated by means of~$|\Om| \ge |B_1| r_\Om^{n} \ge |B_1| R^n$.
\end{proof}

\section*{Acknowledgements} 

All the authors are members of the Australian Mathematical Society (AustMS).
Serena Dipierro is supported by the Australian Future Fellowship FT230100333 ``New perspectives on nonlocal equations''.
Giorgio Poggesi is supported by the Australian Research Council (ARC) Discovery Early Career Researcher Award (DECRA) DE230100954 ``Partial Differential Equations: geometric aspects and applications''.
Jack Thompson is supported by an Australian Government Research Training Program Scholarship.
Enrico Valdinoci is supported by the Australian Laureate Fellowship FL190100081 ``Minimal surfaces, free boundaries and partial differential equations''.

It is a pleasure to thank Lorenzo De Gaspari, João Gonçalves da Silva, and Riccardo Villa for very interesting discussions.

\printbibliography
\vfill
\end{document}